\documentclass[a4paper, reqno]{amsart}

\usepackage{amsmath, amssymb, amsthm}

\usepackage[hidelinks]{hyperref}
\usepackage{xpatch}
\usepackage{tikz-cd}
\usepackage{enumitem}
\usepackage{theoremref}
\usepackage{color}

\usepackage[inline,final]{showlabels}

\showlabels{thlabel}

\setlist[description]{font=\normalfont\scshape}

\xpatchcmd{\proof}{\itshape}{\normalfont\bfseries}{}{}

\newtheorem{theorem}{Theorem}[section]
\newtheorem{proposition}[theorem]{Proposition}
\newtheorem{prop}[theorem]{Proposition}
\newtheorem{lemma}[theorem]{Lemma}
\newtheorem{corollary}[theorem]{Corollary}
\newtheorem{fact}[theorem]{Fact}
\newtheorem{conjecture}[theorem]{Conjecture}
\newtheorem*{thm}{Theorem}

\theoremstyle{definition}
\newtheorem{definition}[theorem]{Definition}
\newtheorem{construction}[theorem]{Construction}
\newtheorem{remark}[theorem]{Remark}
\newtheorem{remarks}[theorem]{Remarks}

\newtheorem{example}[theorem]{Example}
\newtheorem{examples}[theorem]{Examples}
\newtheorem{question}[theorem]{Question}

\newcommand{\calC}{\ensuremath{\mathcal{C}}}

\newcommand{\ExpF}{\ensuremath{\mathbf{ExpF}}}

\newcommand{\EAF}{\ensuremath{\mathbf{EAF}}}
\newcommand{\ELAF}{\ensuremath{\mathbf{ELAF}}}

\newcommand{\ELAFstrong}{\ensuremath{\mathbf{ELAF}_\strong}}

\newcommand{\ExpFkp}{\ensuremath{\ExpF_{\mathrm{kp}}}}
\newcommand{\ELAFkp}{\ensuremath{\ELAF_{\mathrm{kp}}}}
\newcommand{\ELAFKkp}{\ensuremath{\ELAF_{K,\mathrm{kp}}}}

\newcommand{\ECF}{\ensuremath{\mathbf{ECF}}}

\newcommand{\N}{\ensuremath{\mathbb{N}}}
\newcommand{\Z}{\ensuremath{\mathbb{Z}}}
\newcommand{\Q}{\ensuremath{\mathbb{Q}}}

\newcommand{\alg}{\ensuremath{\mathrm{alg}}}

\DeclareMathOperator{\gtp}{gtp}

\DeclareMathOperator{\cl}{cl}
\DeclareMathOperator{\ecl}{ecl}
\DeclareMathOperator{\td}{td}
\DeclareMathOperator{\etd}{etd}
\DeclareMathOperator{\ldim}{ldim}
\DeclareMathOperator{\linspan}{span}

\newcommand{\Cexp}{\ensuremath{\mathbb{C}_{\exp}}}
\newcommand{\Rexp}{\ensuremath{\mathbb{R}_{\exp}}}

\DeclareMathOperator {\Th} {Th}

\DeclareMathOperator {\A} {A}

\DeclareMathOperator {\seq} {\subseteq}

\newcommand{\abar}{\ensuremath{\bar{a}}}
\newcommand{\bbar}{\ensuremath{\bar{b}}}

\newcommand{\ELA}{{\textup{ELA}}}
\newcommand{\EA}{{\textup{EA}}}
\renewcommand{\phi}{\varphi}
\newcommand{\hull}[1]{\lceil #1 \rceil}
\newcommand{\Qlin}{{\Q\textup{-lin}}}
\newcommand{\gen}[1]{\ensuremath{\left\langle #1 \right\rangle}}

\newcommand{\calB}{\mathcal{B}}
\newcommand{\Bexp}{\mathbb{B}_\mathrm{exp}}

\newcommand{\iso}{\cong}

\def\Ind#1#2{#1\setbox0=\hbox{$#1x$}\kern\wd0\hbox to 0pt{\hss$#1\mid$\hss}
\lower.9\ht0\hbox to 0pt{\hss$#1\smile$\hss}\kern\wd0}
\def\ind{\mathop{\mathpalette\Ind{}}}
\def\Notind#1#2{#1\setbox0=\hbox{$#1x$}\kern\wd0\hbox to 0pt{\mathchardef
\nn="3236\hss$#1\nn$\kern1.4\wd0\hss}\hbox to 0pt{\hss$#1\mid$\hss}\lower.9\ht0
\hbox to 0pt{\hss$#1\smile$\hss}\kern\wd0}
\def\nind{\mathop{\mathpalette\Notind{}}}

\newcommand{\into}{\hookrightarrow}

\newcommand{\strong}{\ensuremath{\lhd}}

\renewcommand{\leq}{\leqslant}
\renewcommand{\geq}{\geqslant}

\newcommand{\leteq}{\ensuremath{:=}}

\title{Independence relations for exponential fields}
\author{Vahagn Aslanyan, Robert Henderson, Mark Kamsma and Jonathan Kirby}

\address[All authors]{School of Mathematics, University of East Anglia, 
UK}

\address[Current addresses]{~} 

\email{V.Aslanyan@leeds.ac.uk}
\urladdr{https://vahagn-aslanyan.github.io\vspace{-.7em}}
\address{School of Mathematics, University of Leeds,
UK}

\address{}
\email{roberthenderson@maynard.co.uk}

\address{}

\email{mark@markkamsma.nl}
\urladdr{https://markkamsma.nl\vspace{-.7em}}
\address{Department of Mathematics, Imperial College London, 
UK}
\address{}

\email{jonathan.kirby@uea.ac.uk}

\date{\today, version 1.0
\\
VA and JK were supported by EPSRC grant EP/S017313/1. RH and MK were supported by PhD studentships from UEA. MK was also supported by EPSRC grant EP/W522314/1.}

\subjclass{03C45, 03C65, 03C48}
\keywords{Independence relation, exponential field, abstract elementary class, Ax-Schanuel}

\begin{document}

\maketitle

\begin{abstract}
We give four different independence relations on any exponential field. Each is a canonical independence relation on a suitable Abstract Elementary Class  of exponential fields, showing that two of these are NSOP$_1$-like and non-simple, a third is stable, and the fourth is the quasiminimal pregeometry of Zilber's exponential fields, previously known to be stable (and uncountably categorical). We also characterise the fourth independence relation in terms of the third,~strong independence.
\end{abstract}

\setcounter{tocdepth}{1}
\tableofcontents

\section{Introduction}

\subsection{Independence relations in model theory}

Ternary independence relations are very widely used across model-theory, both in pure model theory, where they arise for instance from Shelah's key notions of splitting and forking, and in applications, where they often capture useful algebraic information. The basic examples include disjointness of subsets, linear independence in vector spaces, and algebraic independence in fields. These are all strongly minimal examples, but independence relations are also important higher up in the stability hierarchy.

Kim and Pillay \cite{kim_simple_1997} proved that if a complete first-order theory $T$ admits an independence relation $\ind$ satisfying a certain list of properties then $T$ lies in the stability class known as simple theories. Furthermore, $\ind$ is the unique independence relation satisfying those properties and is given by non-forking. There is a similar theorem with a slightly stronger list of properties characterizing stable theories, and more recently \cite{kaplan_kim-independence_2020} an analogous theorem for NSOP$_1$-theories.

There have been various generalisations of these Kim--Pillay-style theorems beyond the first-order context, for example to positive logic in \cite{ben-yaacov_simplicity_2003, dobrowolski_kim-independence_2022}, to some Abstract Elementary Classes (AECs) in \cite{BuechlerLessman2003, HyttinenLessman2006, hyttinenKesala_2006}, and to an even more general and abstract context of Abstract Elementary Categories by the third author in \cite{kamsma_kim-pillay_2020, kamsma_nsop_1-like_2022}. There is also other recent work on abstract independence relations in a category-theoretic context in the stable case in \cite{lieberman_forking_2019}.

\subsection{The main theorems}

In this paper we illustrate the theory of independence relations with four examples in exponential fields. None of our examples fit the setting of a complete first-order theory, but they all fit into the context of AECs.

\begin{definition}
\thlabel{def:exponential-field} \thlabel{def:ea-field, ELA-field}
An \emph{exponential field}, or \emph{E-field} for short, is a field $F$ of characteristic zero together with a group homomorphism $\exp: \langle F;+\rangle \to \langle F^\times;\times\rangle$, from the additive group to the multiplicative group of $F$. We will also write $e^x$ instead of $\exp(x)$, or write $\exp_F(x)$ if we need to specify $F$.

We call an E-field $F$ an \emph{EA-field} if the underlying field is algebraically closed. If, in addition, 
every nonzero element has a logarithm (that is, for every $b \in F^\times$ there is $a \in F$ such that $e^a = b$) then we say $F$ is an \emph{ELA-field}.
\end{definition}
The obvious examples of exponential fields are the real and complex fields with exponentiation given by the usual power series, but one can also construct exponential maps algebraically.
See \cite{kirby_finitely_2013} for a detailed account of such constructions.

The four independence relations in this paper are: EA-independence, ELA-independence, strong independence, and the independence relation associated with the exponential algebraic closure pregeometry, and its dimension notion called exponential transcendence degree. We denote these respectively by $\ind^\EA$, $\ind^\ELA$, $\ind^\strong$, and $\ind^{\etd}$. We next explain our main results,  deferring the definitions to later.

EA-independence was introduced in \cite{haykazyan_existentially_2021}, where it was shown to satisfy certain properties with respect to the category of existentially closed exponential fields, and those properties were shown to be sufficient that the associated theory in positive logic is NSOP$_1$, that is, no formula has SOP$_1$.

Subsequently, \cite{kamsma_nsop_1-like_2022} gave a slightly stronger list of properties for an NSOP$_1$-like independence relation, sufficient to rule out the existence of a distinct simple or stable independence relation, which is summarised in \thref{fact:canonicity-of-independence}. In particular, this implies canonicity for simple and stable independence relations. Those stronger properties have been verified in existing literature \cite{haykazyan_existentially_2021, dobrowolski_kim-independence_2022}. We make the addition that a further property fails, meaning that EA-independence cannot be simple, and so there cannot be a simple independence relation.
\begin{theorem}
\thlabel{thm:ea-nsop1-like-independence}
The independence relation $\ind^\EA$ is an NSOP$_1$-like, non-simple independence relation on the category \EAF\ of EA-fields. 
\end{theorem}

The ELA-independence relation is introduced in this paper, as particularly relevant where we consider extensions of exponential fields where the kernel of the exponential map does not extend. We prove:
\begin{theorem}\thlabel{ELA-indep is NSOP1}
For any kernel type $K$, the independence relation $\ind^\ELA$ is  NSOP$_1$-like and non-simple on the category $\ELAFKkp$ of ELA-fields with kernel type $K$ and kernel-preserving embeddings. 
\end{theorem}

Strong embeddings of exponential fields are those which preserve the transcendence properties given by the Ax-Schanuel theorem, and they are particularly important for analytic exponential fields such as $\Rexp$ and $\Cexp$, and also for exponential fields of power series. Zilber's exponential field $\Bexp$ is constructed by amalgamating strong extensions.
The PhD thesis of the third author \cite{henderson_independence_2014} introduced strong independence and proved that it satisfied the properties of a stable independence relation given by Hyttinen and Kes\"al\"a \cite{hyttinenKesala_2006}. Here we publish these results for the first time, updated for the list of properties from \cite{kamsma_kim-pillay_2020}.
\begin{theorem}\thlabel{strong indep is stable}
The strong independence relation $\ind^\strong$ is the canonical independence relation on the category $\ELAF_{\mathrm{vfk},\strong}$ of ELA-fields with very full kernel and strong embeddings, and it is stable. 
\end{theorem}
We would like to remove the restriction in this theorem to exponential fields with \emph{very full kernel}. This is a partial saturation condition, and in particular it implies that the kernel of the exponential map has size at least continuum. The most interesting exponential fields (at least in this context where the fields are algebraically closed, not ordered) have cyclic kernel as in the complex case, so certainly countable kernel. We expect them to sit in stable categories as well.

\begin{conjecture}
Theorem~\ref{strong indep is stable} holds for the category $\ELAF_\strong$, without the assumption of very full kernel.
\end{conjecture}

The exponential-algebraic closure operator was proved to be a pregeometry on any exponential field in \cite{kirby_exponential_2010}. (It was known in the real case earlier.) 
It is the quasiminimal pregeometry on Zilber's exponential field $\Bexp$, and on the quasiminimal excellent class (a type of AEC) used to construct it. It follows that the associated independence relation $\ind^{\etd}$ is a stable independence relation on that AEC.
In this paper we show that $\ind^{\etd}$ is closely related to $\ind^\strong$ on \emph{any} exponential field:
\begin{theorem}
\thlabel{thm:ecl-independence-vs-ecf-independence}
Let $F$ be an exponential field and $A, B, C \subseteq F$. Then we have
\[
A \ind_C^{\etd, F} B \quad  \Longleftrightarrow \quad A \ind_{\ecl_F(C)}^{\strong, F} B.
\]
\end{theorem}

We have categories of exponential fields which are stable and which are NSOP$_1$, non-simple. A natural question which we have not managed to answer is:
\begin{question}
Is there a category of exponential fields which is simple, unstable?
\end{question}

\subsection{Overview of the paper}
In section~2 we give the background on independence relations, and the Kim--Pillay style theorems in our context of AECs. 
Section~3 introduces the four types of embeddings of exponential fields we use: general embeddings, those which preserve the kernel, strong, and closed embeddings. We show that the various categories produced are AECs with the important properties of amalgamation, joint embedding, and intersections.

The independence relations $\ind^\EA$ and $\ind^\ELA$ are defined and compared in sections~4 and~5, and Theorems~\ref{thm:ea-nsop1-like-independence} and~\ref{ELA-indep is NSOP1} are proved there.

In section~6, we define strong independence and prove Theorem~\ref{strong indep is stable}.
Finally, Theorem~\ref{thm:ecl-independence-vs-ecf-independence} is proved in section~7.

\subsection{Categories versus monster models}
Both in the classical setting of complete first-order theories, and when working with AECs, model theorists often use the ``Monster model convention'', that all models considered are submodels of a suitably large saturated ``monster'' model. We do not do that, but take the more algebraic approach of instead working directly with categories of exponential fields. Given that our categories have amalgamation, this change is really one of emphasis rather than being substantial, but it makes several things more convenient.

We take care to separate the properties of an independence relation which apply to an individual structure (in this paper an exponential field), those properties which relate to embeddings of structures (\textsc{Invariance}), and the properties which relate to a category of structures (or in the classical setting, to the common complete theory of the structures). An exponential field may lie in different categories with different independence relations and incompatible monster models, but our approach allows us to make sense of all four independence relations on any exponential field, and so to compare them.

Another idea we try to stress is the close relationship between independence relations and amalgamation, and particularly free amalgamation. This idea is somewhat hidden by the monster model convention.

Thirdly, characterising a monster model of a theory (or of an AEC) involves classifying (and perhaps axiomatising) the existentially closed models. Although we can do this for our AECs, we realised that existential closedness plays no role here, so the models in our categories are not existentially closed, although they usually satisfy some much weaker closure condition related to amalgamation. This highlights an algebraic side to the idea of independence relations, and indeed we make sense of these categories being stable, simple, or NSOP${}_1$-like, with no reference to the existentially closed models or to any theory axiomatising them.\\
\\
\noindent
\textbf{Acknowledgements.} We thank the anonymous referee whose suggestions improved the presentation of this paper.

\section{Independence relations and the stability hierarchy}

In this section we set out our model-theoretic conventions, notation, and terminology for independence relations in a category of structures.

\subsection{Independence relations on a structure}

\begin{definition}
\thlabel{def:independence-relation}
Let $M$ be any structure.  An \emph{independence relation} on $M$ is a ternary relation $\ind$ on subsets of $M$, which satisfies the six basic properties listed below. If $(A,B,C)$ is in the relation we say that \emph{$A$ is independent from $B$ over $C$ in $M$} and write
\[
A \ind_C^M B \qquad \text{ or just \quad} A \ind_C B \quad \text{ if $M$ is clear from the context}.
\]
In the properties below, and throughout the paper, we use the standard model-theoretic convention that juxtaposition of sets or tuples means union or concatenation. For example, $BC$ means $B \cup C$.

{\noindent\bf Basic properties} For all $A, B, C, D \subseteq M$ we have:
\begin{description}
\item[\textsc{Normality}] If $A \ind_C B$ then $A \ind_C BC$.

\item[\textsc{Existence}] $A \ind_C C$.

\item[\textsc{Monotonicity}] If $A \ind_C B$ and $D \subseteq B$ then $A \ind_C D$.

\item[\textsc{Transitivity}] If $A \ind_C D$ and $A \ind_D B$ with $C \subseteq D$ then $A \ind_C B$.

\item[\textsc{Symmetry}] If $A \ind_C B$ then $B \ind_C A$.

\item[\textsc{Finite Character}] If for all finite $D \subseteq A$ we have $D \ind_C B$  then $A \ind_C B$.
\end{description}
\end{definition}
\begin{definition}

One additional property which will often hold is
\begin{description}
\item[\textsc{Base-Monotonicity}] If $A \ind_C B$ and $C \subseteq D \subseteq B$ then $A \ind_D B$.
\end{description}
\end{definition}

\begin{examples}
If $\cl$ is any pregeometry on $M$, with associated dimension function $\dim$, then it is well-known (and easy to verify) that defining \
\[A\ind^{\dim}_C B  \text{\quad  if and only if   for every finite $D \subseteq A$ }, \dim(D/BC) = \dim(D/C)\]
gives an independence relation on $M$ satisfying \textsc{Base-Monotonicity}. 

In particular, on a $\Q$-vector space $M$ we have $A \ind_C^\Qlin B$ if the following equivalent conditions hold:
\begin{enumerate}[label=(\roman*)]
\item for every finite $D \subseteq A$ we have $\ldim_{\Q}(D/BC) = \ldim_{\Q}(D/C)$;
\item $\linspan(AC) \cap \linspan(BC) = \linspan(C)$.
\end{enumerate}
Here and throughout the paper, $\linspan(A)$ will always mean the $\Q$-linear span of $A$, in some ambient $\Q$-vector space (often a field) which will be clear from context.

On any field $F$, we have \emph{field-theoretic algebraic independence} $\A \ind_C^{\td} B$ where the dimension notion is transcendence degree, $\td$, and the pregeometry is (field-theoretic) relative algebraic closure.
\end{examples}

In any exponential field, there is an exponential algebraic closure pregeometry, with dimension notion called exponential transcendence degree. Unlike field-theoretic algebraic closure, the definition is not quantifier-free and cannot be reduced to one variable at a time, but comes from an algebraic version of the implicit function theorem.

\begin{definition}
Let $F$ be any exponential field.

We say $a_1 \in F$ is \emph{exponentially algebraic} over a subset $B \subseteq F$ iff for some $n \in \N$ there are:
$\bar{a} = (a_1,\ldots,a_n) \in F^n$,  polynomials $p_1,\ldots,p_n \in \Z[\bar{X},e^{\bar{X}},\bar Y]$ , and a tuple $\bbar$ from $B$
such that setting $f_i(\abar) = p_i(\abar,e^{\abar},\bbar)$ we have
\begin{itemize}
 \item $f_i(\bar{a}) = 0$ for each $i=1,\ldots,n$, and
\item $\begin{vmatrix}
\frac{\partial f_1}{\partial X_1} & \cdots & \frac{\partial f_1}{\partial X_n} \\
\vdots &\ddots & \vdots \\
\frac{\partial f_n}{\partial X_1} & \cdots & \frac{\partial f_n}{\partial X_n} \\
\end{vmatrix}(\bar{a}) \neq 0$.
\end{itemize}
where $\frac{\partial}{\partial X_i}$ denotes the formal partial differentiation of exponential polynomials.

Otherwise, $a_1$ is \emph{exponentially transcendental} over $B$ in $F$.

We write $\ecl_F(B)$ for the exponential-algebraic closure of $B$ in $F$. It is always an exponential subfield, (field-theoretically) relatively algebraically closed in $F$, and closed under any logarithms which exist in $F$.
\end{definition}

By \cite[Theorem~1.1]{kirby_exponential_2010}, exponential-algebraic closure is a pregeometry on any exponential field. The associated dimension notion is known as \emph{exponential transcendence degree}, and we denote the associated independence relation by $A \ind_C^{\etd} B$.

\subsection{Independence relations on categories of structures}
 In model theory, we usually define independence relations not just on one model, but on all models of a complete theory, and require them to be compatible under taking elementary extensions. In this paper we will generalise this approach by
\begin{enumerate}
\item working with models of a theory which may not be complete, and which may not be defined in any particular logic, and
\item by specifying which extensions we consider, not just elementary extensions.
\end{enumerate}

We consider concrete categories of structures, meaning categories in which every object has an underlying set, and every arrow has an underlying function which determines the arrow.
In this paper, the objects will always be exponential fields, with the arrows being embeddings of exponential fields, sometimes with additional restrictions.

\begin{definition}
Let $\calC$ be a concrete category of structures. An independence relation on $\calC$ consists of an independence relation $\ind^M$ for each object $M \in \calC$, which together satisfy:
\begin{description}
\item[\textsc{Invariance}] For any $f: M \to N$ in $\calC$ and any subsets $A,B,C \subseteq M$ we have 
\[A \ind_C^M B \quad \text{ iff } \quad f(A) \ind_{f(C)}^N f(B).\]
\end{description}
\end{definition}
The categories of structures and extensions we consider will all have amalgamation and unions of chains, so we can construct monster models in them in any of the usual ways. Our definition of \textsc{Invariance} is then equivalent to the common definition of invariance under automorphisms of the monster model.

\subsection{Abstract elementary classes}
In \cite{kamsma_kim-pillay_2020, kamsma_nsop_1-like_2022}, independence relations were developed in the very general setting of \emph{Abstract Elementary Categories (AECats)}, a class of accessible categories which are not required to be concrete. They are a generalisation of Shelah's notion of Abstract Elementary Class (AEC), which itself generalises categories of models of theories in a wide range of logics. All the examples we will consider in this paper are AECs, so we define those (albeit in more category-theoretic language than Shelah's original definition).

\begin{definition}
An \emph{abstract elementary class (AEC)} is a category $\calC$ such that for some first-order vocabulary $L$, every object is an $L$-structure (which we call a \emph{model} in $\calC$) and every arrow is an $L$-embedding, satisfying the following properties:
\begin{enumerate}
\item $\calC$ is closed under isomorphisms: if $A \in \calC$ and $f: A \cong B$ is an $L$-isomorphism then $B$ and $f$ are in $\calC$.
\item Coherence: If $A\subseteq B \subseteq C$ are objects in $\calC$ with the inclusions $A \into C$ and $B \into C$ both in $\calC$, then also the inclusion $A \into B$ is in $\calC$.
\item $\calC$ is closed under unions of chains: for any ordinal $\lambda$, if $(A_i)_{i<\lambda}$ are in $\calC$ such that for all $i<j<\lambda$ we have $A_i \subseteq A_j$ with the inclusion functions in $\calC$, then $A \leteq \bigcup_{i< \lambda}A_i \in \calC$ and all inclusions $A_i \into A$ are in $\calC$. Furthermore, if all $A_i \subseteq B$ with inclusion maps in $\calC$ then the inclusion $A \into B$ is also in $\calC$. (It is a standard consequence that $\calC$ is then also closed under unions of directed systems \cite[Corollary 1.7]{adamek_locally_1994}.)

\item The Downwards L\"owenheim--Skolem property: There is an infinite cardinal $\kappa$ (the smallest such being called the LS-cardinal of $\calC$), such that for every $A \in \calC$ and every subset $S\subseteq A$, there is a subobject $B\into A$ such that $S \subseteq B$ and $|B| \leqslant |S| + \kappa$.
\end{enumerate}
\end{definition}

We will be considering AECs which have amalgamation and intersections in the sense below, in most cases by choosing the objects to be exactly the amalgamation bases from a larger category.
\begin{definition}
An object $A$ in a category $\calC$ is said to be an \emph{amalgamation base} if for every pair of arrows $B \leftarrow A \to C$ there are arrows $B \to D \leftarrow C$ making the relevant square commute.

A category $\calC$ is said to have the \emph{amalgamation property (AP)}, or be a category \emph{with amalgamation}, if every object is an amalgamation base.

A category $\calC$ has the \emph{Joint Embedding Property (JEP)} if for every two objects $A,B$, there is an object $D$ and arrows $A \to D \leftarrow B$. In the presence of AP, having such a common extension is an equivalence relation on the objects in the category. We call the equivalence classes of this relation \emph{JEP-classes}.

We say that an AEC $\calC$ \emph{has intersections} if for any object $A$, and any set $(S_i)_{i \in I}$ of subobjects of $A$, the intersection $\bigcap_{i\in I} S_i$ is also a subobject of $A$.
\end{definition}

\begin{definition}
Let $\calC$ be an AEC with amalgamation, and let $M_1$ and $M_2$ be models in $\calC$. Possibly infinite tuples $a_1 \in M_1$ and $a_2 \in M_2$ are said to have the same \emph{Galois type} if there is a model $N$ and embeddings $g_i:M_i \into N$, in $\calC$ such that $g_1(a_1) = g_2(a_2)$. Using amalgamation it is easy to see that this gives an equivalence relation on pairs $(a;M)$. We write $\gtp(a; M)$ for the Galois type (the equivalence class).

We can also define Galois types over sets of parameters as a special case. Suppose that $a_i = b_ic$ for $i=1,2$, where $c$ is a tuple from $M_1$ and $M_2$. Then we write $\gtp(b_1/c ; M_1) = \gtp(b_2/c ; M_2)$ to mean $\gtp(b_1c ; M_1) = \gtp(b_2c ; M_2)$.

If $C$ is the common subset of $M_1$ and $M_2$ enumerated by $c$, we also write this as $\gtp(b_1/C ; M_1) = \gtp(b_2/C ; M_2)$.
\end{definition}
Note that if $M \into N$ is an extension of models in $\calC$ and $a \in M$ then we always have $\gtp(a;N) = \gtp(a;M)$, so where no confusion is likely to occur we will drop the $M$ from the notation and just write $\gtp(a)$.

There is a simple characterisation of Galois types in AECs with amalgamation and intersections.
\begin{lemma}\thlabel{Galois types from intersections}
Suppose that $\calC$ is an AEC with amalgamation and intersections. Let  $f_1: C \into A$ and $f_2: C \into B$ be embeddings in $\calC$, and let $a \in A$ and $b \in B$ be tuples. Let $[Ca]$ be the intersection of all the subobjects of $A$ containing $C \cup a$, and likewise $[Cb]$. Then $\gtp(a/C) = \gtp(b/C)$ if and only if there is an isomorphism $[Ca] \iso [Cb]$ fixing $C$ pointwise and taking $a$ to $b$.
\end{lemma}
\begin{proof}
Straightforward from the definitions.
\end{proof}

\begin{remark}
Note that if $\calC$ is an AEC with amalgamation and $\ind$ is an independence relation on $\calC$, then the \textsc{invariance} property is equivalent to saying that if $A \ind^M_C B$ and we have $A',B',C' \subseteq M'$ such that (for any choice of enumerations of $A$, $B$, $C$, $A'$, $B'$, and $C'$ as  tuples) $\gtp(A B  C;M) =  \gtp(A'  B'  C';M')$  then $A' \ind_{C'}^{M'} B'$.
\end{remark}

\subsection{The independence relation hierarchy}

To give the hierarchy of stable, simple, and NSOP$_1$-like independence relations, we consider additional properties for an independence relation on an AEC with amalgamation. We first recall the definition of a \emph{club set} in a suitable part of a powerset.  

\begin{definition}
\thlabel{def:club-set}
Let $\lambda$ be a regular cardinal and $X$ any set. We write $[X]^{< \lambda} = \{ Y \subseteq X : |Y| < \lambda \}$. We call a family of subsets $\calB \subseteq [X]^{< \lambda}$:
\begin{itemize}
\item \emph{unbounded} if for every $Z \in [X]^{< \lambda}$ there is $Y \in \calB$ with $Z \subseteq Y$.
\item \emph{closed} if for every chain $(Y_i)_{i < \gamma}$ in $\calB$ (i.e. $i \leq j < \gamma$ implies $Y_i \subseteq Y_j$) with $\gamma < \lambda$ we have that $\bigcup_{i < \gamma} Y_i \in \calB$.
\item a \emph{clubset} if $\calB$ is closed and unbounded.
\end{itemize}
\end{definition}

\begin{definition}[Additional properties for an independence relation] \ \\
\thlabel{def:independence-relation-advanced-properties}
The tuples below are allowed to be infinite.
\begin{description}

\item[\textsc{Club Local Character}] There is a cardinal $\lambda$ such that for any model $M$ in $\calC$, any finite subset $A \subseteq M$ and any subset $B \subseteq M$ there is a clubset $\calB \subseteq [B]^{< \lambda}$ such that $A \ind_{B_0}^M B$ for all $B_0 \in \calB$.

\item[\textsc{Extension}] If $a \ind_C^M B$ and $B \subseteq B' \subseteq M$ then there is an extension $M \into N$ in $\calC$ and $a' \in N$ such that $a' \ind_C^N B'$ and $\gtp(a'/BC; N) = \gtp(a/BC; M)$.

\item[\textsc{3-amalgamation}] Suppose we are given a commuting diagram in $\calC$ consisting of the solid arrows below
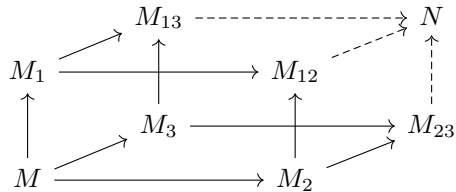
\[
\begin{tikzcd}[row sep=tiny]
                                   & M_{13} \arrow[rr, dashed] &                           & N                         \\
M_1 \arrow[ru] \arrow[rr]          &                           & M_{12} \arrow[ru, dashed] &                           \\
                                   & M_3 \arrow[uu] \arrow[rr] &                           & M_{23} \arrow[uu, dashed] \\
M \arrow[ru] \arrow[rr] \arrow[uu] &                           & M_2 \arrow[uu] \arrow[ru] &                          
\end{tikzcd}
\]
Suppose furthermore that $M_1 \ind^{M_{12}}_M M_2$, $M_2 \ind^{M_{23}}_M M_3$ and $M_3 \ind^{M_{13}}_M M_1$. Then we can find $N$ together with the dashed arrows, making the diagram commute and such that $M_1 \ind_M^N M_{23}$.

\item[\textsc{Stationarity}] Let $M \subseteq N$ be models in $\calC$. If we have $a_1 \ind_M^N B$, $a_2 \ind_M^N B$ and $\gtp(a_1/M; N) = \gtp(a_2/M; N)$ then $\gtp(a_1/MB; N) = \gtp(a_2/MB; N)$.
\end{description}
\end{definition}

\begin{definition}
\thlabel{def:independence-relations}
Suppose that $\ind$ is an independence relation on an AEC with amalgamation $\calC$. We say that $\ind$ is:
\begin{itemize}
\item an \emph{NSOP$_1$-like independence relation} if it also satisfies \textsc{Club Local Character}, \textsc{Extension} and \textsc{3-amalgamation};
\item a \emph{simple independence relation} if in addition it satisfies \textsc{Base-Monotonicity};
\item a \emph{stable independence relation} if in addition it satisfies \textsc{Stationarity}.
\end{itemize}
\end{definition}
In particular we have for an independence relation that being stable implies being simple implies being NSOP$_1$-like.

\begin{remarks}
\thlabel{indep-properties-remark}
~
\begin{enumerate}
\item The usual formulation of \textsc{Local Character} requires some cardinal $\lambda$ such that for all $A, B \subseteq M$ where $A$ is finite there is some $B_0 \subseteq B$ with $|B_0| < \lambda$ such that $A \ind_{B_0}^M B$. In the presence of \textsc{Base-Monotonicity} this implies \textsc{Club Local Character}, by considering the clubset
\[
\{ B_1 \subseteq B : |B_1| < \lambda, B_0 \subseteq B_1 \}.
\]
In NSOP$_1$-like independence relations the property \textsc{Base-Monotonicity} may not hold, but one insight of \cite{kaplan_local_2019} is that \textsc{Club Local Character} captures what is necessary for applications.

\item It is well known for classical first-order logic that the \textsc{3-amalgamation} property follows from the rest of the properties of a stable independence relation. For a proof covering the generality of AECs, see \cite[Proposition 6.16]{kamsma_kim-pillay_2020}.

\item Our formulation of 3-amalgamation is at first sight slightly different from that in \cite{kamsma_kim-pillay_2020, kamsma_nsop_1-like_2022}: there $M_1$, $M_2$ and $M_3$ would not necessarily be models and $M$ would not necessarily factor through them. However, modulo the basic properties in \thref{def:independence-relation} together with a repeated application of \textsc{Extension} the two versions are easily seen to be equivalent.

\item For a complete first-order theory $T$, if there is a simple or stable independence relation such that the cardinal $\lambda$ for local character is $\aleph_0$ then the theory is supersimple or superstable respectively. We will show in \thref{prop:strong-local-character} that our notion of strong independence has local character with cardinal $\aleph_0$. However these notions of superstability and supersimplicity are not so well-developed beyond the first-order setting so we do not immediately get any further conclusions.

\item The hierarchy of NSOP$_1$-like --- simple --- stable can be extended by adding \emph{stable and coming from a pregeometry} (such as the quasiminimal case, or the uncountably categorical case) but that does not seem to correspond to axioms on the independence relation in the same style.
\end{enumerate}
\end{remarks}

\begin{examples}
The $\ind^\Qlin$ relation defined earlier satisfies \textsc{invariance} for embeddings of $\Q$-vector spaces, and is well-known to give a stable independence relation on the category of $\Q$-vector spaces and their embeddings.

More generally, if $T$ is a strongly minimal theory then the independence relation coming from its pregeometry is a stable (indeed superstable) independence relation.

The independence relation $\ind^{\td}$ on a field satisfies \textsc{invariance} for field embeddings and gives a stable independence relation on the category of fields and field embeddings. This is almost a strongly minimal example; the additional content is that there is no need to mention algebraically closed fields, or to fix the characteristic.
\end{examples}

The following fact tells us that there can be at most one nice enough independence relation on an AECat.
\begin{fact}[Canonicity of independence, {\cite[Theorem 1.3]{kamsma_nsop_1-like_2022}}]
\thlabel{fact:canonicity-of-independence}
Let $\calC$ be an AEC with the amalgamation property and suppose that $\ind$ is a stable or a simple independence relation on $\calC$. Suppose furthermore that $\ind^*$ is an NSOP$_1$-like independence relation on $\calC$. Then $\ind = \ind^*$ over models. That is for $M \into N$ in $\calC$ and $ A, B \subseteq N$ we have $A \ind_M^N B$ iff $A \ind_M^{*, N} B$.
\end{fact}
This statement does not allow for comparing two NSOP$_1$-like independence relations. This is intentional, because in order to do so with current knowledge, an extra assumption called the ``existence axiom for forking'' is required. The current statement avoids this, because having a simple independence relation implies the existence axiom for forking. We also only get the result over models, rather than arbitrary sets because we only require \textsc{3-amalgamation} rather than the full property called ``independence theorem'' in \cite{kamsma_nsop_1-like_2022}. However, the current statement is enough for our purposes.

\section{Categories of exponential fields}

\subsection{Categories with all exponential field embeddings}

We write $\ExpF$ for the category of exponential fields, with embeddings as arrows. We write $\EAF$ and $\ELAF$ for the full subcategories of EA-fields and ELA-fields.

\begin{prop}
\thlabel{prop:basic-categories-aecs}
The categories $\ExpF$, $\EAF$, and $\ELAF$ are AECs, both  $\EAF$ and $\ELAF$ have the Amalgamation Property, and $\ExpF$ and $\EAF$ have intersections.
\end{prop}
\begin{proof}
We treat exponential fields as structures in the language $\langle+,\cdot,-,0,1,\exp \rangle$ of exponential rings, with $\exp$ as a unary function symbol. As we are considering all embeddings in this language, coherence and the downward L\"owenheim--Skolem property are immediate. Each category has an $\forall\exists$-axiomatisation in classical first-order logic, so it is closed under isomorphisms and unions of chains.

By Theorem~4.3 of \cite{haykazyan_existentially_2021} (see also the proof of \thref{prop:amalgam-base} below), the amalgamation bases of $\ExpF$ are precisely the EA-fields. As every exponential field extends to an EA-field and to an ELA-field, it follows that both  $\EAF$ and $\ELAF$ have the Amalgamation Property.

It is straightforward to see that $\ExpF$ and $\EAF$ have intersections.
\end{proof}

\begin{definition}
For an EA-field $F$ and a subset $A \subseteq F$ we write $\gen{A}_F^\EA$ for the smallest EA-subfield of $F$ containing $A$. Note that if $F_1 \subseteq F_2$ are both EA-fields and $A \subseteq F_1$ then $\gen{A}_{F_1}^\EA = \gen{A}_{F_2}^\EA$, so we will usually drop the subscript and write just $\gen{A}^\EA$.
\end{definition}

Note that the AEC $\EAF$ does not have JEP, but $F_1$ and $F_2$ lie in the same JEP-class if and only if $\gen{0}^\EA_{F_1} \iso \gen{0}^\EA_{F_2}$. 

In constructions of exponential fields it is often useful to consider the notion of a \emph{partial E-field}:  a field $F$ equipped with a $\Q$-linear subspace $D(F)$ of its additive group and a homomorphism $\exp_F : \langle D(F);+\rangle \to \langle F^\times;\times\rangle$. We consider partial E-fields as structures in the language of rings together with a binary predicate for the graph of the exponential map. 

\begin{definition}
For a partial E-field $F$ and a subset $A \subseteq D(F)$ we write $\gen{A}_F$ for the smallest partial E-subfield of $F$ containing $A$. That is, $\gen{A}_F$ is the field generated by $\linspan(A) \cup \exp(\linspan(A))$ and $D(\gen{A}_F) = \linspan(A)$. If $F_1 \subseteq F_2$ are both partial E-fields and $A \subseteq D(F_1)$ then $\gen{A}_{F_1} = \gen{A}_{F_2}$, so we may drop the subscript and write just $\gen{A}$.
\end{definition}

\begin{construction}[{See \cite[Constructions~2.7,2.9]{kirby_finitely_2013}}]\thlabel{free EA construction}
Let $F$ be a partial E-field. Then there is a \emph{free EA-field extension} $F^\EA$ of $F$, which is obtained from $F$ by taking a point $a \in F^\alg \setminus D(F)$ and adjoining an exponential $e^a$ to $F$, transcendental over $F$, and iterating. One can also get a free (total) E-field extension $F^\textup{E}$ of $F$ the same way, by taking only points $a \in F\setminus D(F)$ at each stage.
These extensions $F^\EA$ and $F^\textup{E}$ can easily be seen to be unique up to isomorphism as extensions of $F$.

The extensions $F^\EA$ and $F^\textup{E}$ of $F$ are free on no generators. One can also get free extensions of $F$ on generators $(x_i)_{i \in I}$ by taking $F_1 = F(x_i)_{i \in I}$, the field of rational functions over $F$, with $D(F_1)=D(F)$ and $\exp_{F_1} = \exp_F$, and then forming the extensions $F_1^\EA$ and $F_1^\textup{E}$.
\end{construction}
Here and in \thref{free ELA construction} we use the term ``free'' because this matches the intuition that no unnecessary algebraic or exponential relations are introduced. However, these constructions are not free in the traditional category-theoretic sense. See \cite[p948]{kirby_finitely_2013} for a further discussion.

\subsection{Kernel-preserving embeddings}

\begin{definition}
By the \emph{kernel} of an exponential field $F$, written $\ker_F$, we mean the kernel of the exponential map $\exp_F$.

We say that $F$ has \emph{standard kernel} if $\ker_F = \tau \Z$, an infinite cyclic group generated by $\tau$ which is transcendental, as in $\Cexp$ where $\tau=2\pi i$.

An embedding $f: F_1 \into F_2$ of exponential fields is \emph{kernel-preserving} if every element of the $\ker_{F_2}$ is in the image of $F_1$. (So the kernel is fixed set-wise, but not necessarily pointwise).

We say that an exponential field $F$ has \emph{full kernel} if it can be embedded in a kernel-preserving way into an ELA-field. (Equivalently, by Proposition~2.12 and Construction~2.13 of \cite{kirby_finitely_2013}, $F$ contains all roots of unity, and they are in the image of $\exp_F$.)
\end{definition}

Much as in \thref{free EA construction}, we can extend a partial E-field with full kernel to an ELA-field in a free way. We give more detail for this construction as we will use it later.

\begin{definition}
Let $F$ be a partial E-field with full kernel. A kernel-preserving partial E-field extension $F'$ is said to be a \emph{one-step free extension} of $F$ if we have $\ldim_
\Q(D(F')/D(F)) = 1$, and, for some (equivalently all) $a \in D(F') \setminus D(F)$ we have either:
\begin{itemize}
\item $a$ is algebraic over $F$ and $e^a$ is transcendental over $F$; or
\item $a$ is transcendental over $F$ and $e^a$ is algebraic over $F$.
\end{itemize}
\end{definition}

\begin{construction}\cite[Construction~2.13]{kirby_finitely_2013}\thlabel{free ELA construction}
Let $F$ be a partial E-field with full kernel, and $M$ an ELA-field extension of $F$ with the same kernel. We say that $M$ is \emph{a free ELA-extension} of $F$ if there is an ordinal-indexed continuous chain of partial E-fields
\[F = F_0 \into F_1 \into \cdots \into F_\alpha \into \cdots \into F_\lambda = M\]
such that each successor step is a one-step free extension.

It is easy to see that free ELA-extensions exist. We denote any such free ELA-extension of $F$ by $F^\ELA$.
\end{construction}
Unlike in the case of $F^\EA$, it is not obvious or even always true that $F^\ELA$ is unique up to isomorphism. For example, take $F = \Q^\alg(2\pi i)$ with $D(F) = 2\pi i\Q$, and $\exp(2\pi i/m)$ a primitive $m^{th}$ root of $1$. Then if we adjoin $a$ transcendental over $F$ such that $\exp(a) = 2$ then the sequence $(\exp(a/m))_{m \in \N^+}$ must be chosen to be one of the continuum-many sequences $(\sqrt[m]2)_{m\in\N^+}$ from $\Q^\alg$, and even allowing the translation $a \mapsto a + \mu$ for a kernel element $\mu$ only allows countably many of the sequences to be realised in a kernel-preserving extension of $F$. We can avoid this if the the kernel is sufficiently saturated in the following sense.

As an abelian group, a full kernel is always a model of $\Th(\Z;+)$. Such groups $M$ are isomorphic to a direct sum $M_r \oplus M_d$ where $M_d \subseteq M$ is the subgroup of divisible elements, and $M_r = M / M_d$ is the \emph{reduced part} of $M$. This reduced part is always an elementary submodel of the profinite completion $\hat\Z$ of $\Z$, see \cite[Chapter~15]{Rothmaler}.

\begin{definition}
A partial E-field has \emph{very full kernel} if the reduced part of its kernel is all of $\hat\Z$.
\end{definition}

\begin{theorem}[Uniqueness of free extensions]
\thlabel{uniqueness of free extensions}
Let $F$ be a partial E-field with very full kernel which is generated as a field by $D(F) \cup \exp(D(F))$. Then $F^\ELA$ is unique up to isomorphism as an extension of $F$.
\end{theorem}
\begin{proof}
This is \cite[Proposition~3.13]{kirby_exponentially_2014}.
\end{proof}
\begin{remark}
\thlabel{rem:uniqueness-countable-case}
It follows from \cite[Theorem~2.18]{kirby_finitely_2013} that the conclusion of \thref{uniqueness of free extensions} also holds when $F$ has full kernel and one of the following holds:
\begin{enumerate}
\item $D(F)$ is finite dimensional, or
\item $D(F)$ is finite dimensional over some countable ELA-subfield (or even just LA-subfield).
\end{enumerate}
 We discuss these cases further in section \ref{subsec:more-general-kernels}.
\end{remark}
\begin{definition}
We write $\ExpFkp$ for the category of exponential fields with full kernel, and kernel-preserving embeddings, and $\ELAFkp$ for the full subcategory of ELA-fields with kernel-preserving embeddings.
\end{definition}
\begin{proposition}\thlabel{prop:amalgam-base}
The amalgamation bases for $\ExpFkp$ are precisely the $\ELA$-fields.
\end{proposition}
\begin{proof}
Let $F$ be an ELA-field and let $f_1:F \to F_1$ and $f_2: F \to F_2$ be two kernel-preserving extensions. We can amalgamate $F_1$ and $F_2$ freely as fields over $F$ and then $\exp_{F_1} \cup \exp_{F_2}$ extends uniquely by additivity to the $\Q$-linear space $F_1 + F_2$. It is easy to check that this does not introduce any new kernel elements. It is then easy to extend this partial E-field to an ELA-field without adding new kernel elements, for example freely as in \thref{free ELA construction}. So ELA-fields are amalgamation bases in $\ExpFkp$, and indeed $\ELAFkp$ has amalgamation.

Conversely, suppose that $F$ is an exponential field with full kernel which is not an ELA-field. First suppose that $F$ is not algebraically closed, and take $a \in F^{\alg} \setminus F$. Using \thref{free ELA construction}, we can form the free ELA-extension $F^{\ELA}$ of $F$ in which the exponentials of $a$ and its conjugates are all transcendental over $F$, and the kernel does not extend. 

We can also form a partial E-field extension $F_1$ of $F$ by choosing a coherent system of roots $(a_m)_{m \in \N^+}$ of $a$ in $F^{\alg}$, that is, we have $a_1=a$ and for all $m,r \in \N^+$ we have $a_{mr}^r = a_m$, and then defining $\exp(la/m + b) = a_m^l \cdot \exp_F(b)$ for all $l \in \Z$, all $m \in \N^+$, and all $b \in F$.

This is a kernel-preserving extension, because if $\exp(la/m + b) = 1$ with $l \neq 0$ then $a_m^l = \exp_F(-b)$, so there is a root of unity $\xi$ such that $\exp_F(-mb/l) = a\xi $. Since $F$ has full kernel, $\xi \in F$, and this contradicts the fact that $a \notin F$.

Now we form the free ELA-extension $F_1^{\ELA}$ of $F_1$. We have two kernel-preserving extensions $F^{\ELA}$ and $F_1^{\ELA}$ of $F$. In $F_1^\ELA$ we have $\exp(a) = a$ so $\exp(a) \in F^\alg$, but if $a'$ is any conjugate of $a$ in $F^\ELA$ then $\exp(a')$ is transcendental over $F$. Hence these extensions cannot be amalgamated over $F$ (even if we allow the kernel to extend).

Now suppose that $F$ is an EA-field, but the exponential map of $F$ is not surjective, say $b\in F^{\times}$ has no logarithm in $F$.

Let $F_1$ be a partial E-field extension of $F$ generated by an element $a$ such that $\exp(a) = b$. Then $a \notin F$ and so $a$ is transcendental over $F$. Then $a^2$ is not in the domain of $\exp_{F_1}$. The image of $\exp_{F_1}$ is the multiplicative span of the image of $\exp_F$ and $b$, so in particular it does not contain $a$. Therefore we can define a further partial E-field extension $F_2$ of $F_1$ with domain spanned by $F$, $a$, and $a^2$, such that $\exp_{F_2}(a^2) = a$. Furthermore, $F_2$ is a kernel-preserving extension of $F$.

Now consider the two extensions $F \into F^\ELA$ and $F \into F_2^\ELA$. If they amalgamate over $F$ without extending the kernel, say into an exponential field $F'$, then the element $a \in F_2$ must map to one of the logarithms of $b$ in $F'$, say $a'$. But this must also come from one of the logarithms of $b$ in $F^\ELA$, which implies that $a'$ and $\exp((a')^2)$ are algebraically independent over $F$, a contradiction.

Hence ELA-fields are the only amalgamation bases in $\ExpFkp$.
\end{proof}
From the proof of \thref{prop:amalgam-base} we also get directly:
\begin{corollary}
\thlabel{cor:independent-ela-amalgamation}
Any span $F_1 \leftarrow F \to F_2$ in $\ExpFkp$ with $F$ an ELA-field can be amalgamated such that $F_1 \ind_F^{\td} F_2$ in the resulting amalgam.\qed
\end{corollary}

\begin{definition}
Let $F$ be an ELA-field, and $A \subseteq F$ a subset. We write $\gen{A}^\ELA_F$ for the intersection of all ELA-subfields $B$ of $F$ containing $A \cup \ker_F$.
\end{definition}
Note that we force $\gen{A}_F^\ELA$ to contain $\ker_F$, so it is not just the intersection of all ELA-subfields. Whenever $F_1 \subseteq F_2$ is a kernel-preserving inclusion of ELA-fields then for any $a \in F_2$ with $\exp(a) \in F_1$ we have $a \in F_1$. It is then easy to see that $\gen{A}^\ELA_F$ is an ELA-subfield of $F$, and hence the category $\ELAFkp$ has intersections. (The category $\ELAF$ actually does not have intersections.)

Furthermore, for any $A \subseteq F_1 \subseteq F_2$ with $\ker_{F_1} = \ker_{F_2}$ we have $\gen{A}^\ELA_{F_1} = \gen{A}^\ELA_{F_2}$, so provided we have fixed the kernel we will usually drop the subscript and write just $\gen{A}^\ELA$.

The JEP-classes of $\ELAFkp$ are given by the isomorphism types of $\gen{0}^\ELA_F$. We call this the \emph{kernel type} of $F$. We write \ELAFKkp\ for the full subcategory of \ELAFkp\ consisting of the ELA-fields with kernel type $K$.

\begin{prop}
Each category $\ELAFKkp$ is an AEC with amalgamation, joint embedding, and intersections.
\end{prop}
\begin{proof}
That $\ELAFKkp$ is an AEC follows from the fact that $\ELAF$ is an AEC (\thref{prop:basic-categories-aecs}), where for the downward L\"owenheim-Skolem property we use the same property for $\ELAF$ where we make sure that $K$ is included in the smaller model. (Thus the LS-cardinal will be $|K|$, and since this is unbounded, it prevents $\ELAFkp$ being an AEC.) The amalgamation property follows from \thref{prop:amalgam-base}, JEP is then immediate, and we have just observed that it is closed under intersections.
\end{proof}

\subsection{Strong embeddings}

In any analytic exponential field, in particular $\Rexp$ and $\Cexp$, or more generally any exponential field where the exponential algebraic closure pregeometry $\ecl$ is non-trivial, the Ax-Schanuel theorem is relevant. It gives non-negativity of a certain predimension function. The embeddings which preserve this predimension function, and in particular preserve its non-negativity, are the \emph{strong embeddings}. Zilber's exponential field $\Bexp$ is constructed by amalgamation of these strong embeddings.

\begin{definition}
\thlabel{def:relative-predimension-over-the-kernel}
\begin{enumerate}
\item 
Let $F$ be a partial E-field. We define the \emph{relative predimension over the kernel} as follows. For a finite tuple $a \in D(F)$ and $B \subseteq D(F)$ we define:
\[
\Delta_F(a/B) \leteq \td(a, \exp(a) / B, \exp(B), \ker_F) - \ldim_{\Q}(a / B, \ker_F).
\]
We omit $B$ if $B$ is empty, so $\Delta_F(a) = \Delta_F(a / \emptyset)$. We may omit the subscript $F$ if the field is clear from the context.

\item An embedding $A \into F$ of partial E-fields is \emph{strong} if it is kernel-preserving and for all finite tuples $b \in D(F)$ we have $\Delta_F(b/A) \geq 0$. We write $A \strong F$ for a strong embedding.

\item If $F$ is a partial E-field and $A$ is a subset of $D(F)$, we say that $A$ is \emph{strong} in $F$ and write $A \strong F$ if for all finite tuples $b$ from $F$ we have $\Delta(b/A) \geq 0$. This agrees with the previous definition in the sense that $A$ is strong in $F$ if and only if the embedding $\gen{A \cup \ker_F}_F \hookrightarrow F$ is strong.

\item It is easy to check that isomorphisms are strong and the composition of strong embeddings is strong, so ELA-fields and strong embeddings form a category which we denote by $\ELAFstrong$.
\end{enumerate}
\end{definition}

\begin{remarks}
\begin{enumerate}
\item When the kernel $K$ is the standard kernel, the predimension function $\Delta$ is for all purposes equivalent to the more commonly used predimension function 
\[\delta(a/B) \leteq \td(a, \exp(a) / B, \exp(B)) - \ldim_{\Q}(a / B).\]
Of course if $B$ contains the kernel then $\delta$ and $\Delta$ agree anyway.

\item The paper \cite{kirby_exponentially_2014} contains an analysis of embeddings for which the predimension inequality holds, but which do not necessarily preserve the kernel, there called \emph{semi-strong} embeddings. The category $\ECF$ of \emph{exponentially closed fields}, conjecturally the category of models of the complete first-order theory of $\Bexp$ and their elementary embeddings, is a further refinement of those ideas, developed in the same paper. Such a theory would interpret the theory of arithmetic and thus has SOP$_1$, so no NSOP$_1$-like independence relation can exist.
\end{enumerate}
\end{remarks}

An exponential field may have no proper strong subsets. For example, this is true for exponential fields which are existentially closed for all embeddings. However, in exponential fields with some proper strong subsets there are many of them and they play an important role as we now explain.

\begin{definition}
\thlabel{def:hull}
Let $F$ be a partial E-field and $A \subseteq D(F)$. We define the \emph{hull} of $A$ in $F$ to be
\[
\hull{A}_F = \bigcap \{B \subseteq D(F) : B \text{ is a $\Q$-linear subspace, } A \cup \ker_F \subseteq B, \text{ and } B \strong F \}.
\]
Note that if $F_1 \strong F_2$ and $A \subseteq D(F_1)$ then $\hull{A}_{F_1} = \hull{A}_{F_2}$, so we omit the subscript $F$ from the notation unless it is needed.
\end{definition}

\begin{lemma}\thlabel{hull finite char}
Let $F$ be a partial E-field and $A \subseteq D(F)$. Then
\begin{enumerate}
\item $\hull{A}$ is well-defined and is strong in $F$.
\item The hull operator has finite character, that is, $\hull{A} = \bigcup_{A_0 \subseteq_{\text{finite}} A} \hull{A_0}$.
\item Suppose that $C \strong F$ and $a$ is a finite tuple from $D(F)$. Then $\ldim_{\Q}(\hull{C a} / C)$ is finite.
\end{enumerate}

\end{lemma}

\begin{proof}
\begin{enumerate}
\item We always have $D(F) \strong F$, so the intersection is non-empty and so well-defined. The fact that it is strong in $F$ is \cite[Lemma~4.5]{bays_pseudo-exponential_2018}.
\item This slightly improves the statement of \cite[Lemma~4.7]{bays_pseudo-exponential_2018}, but the proof is identical: from the definition of the hull, it is immediate that the union $U \leteq \bigcup_{A_0 \subseteq_{\text{finite}} A} \hull{A_0}$ satisfies $A \cup \ker_F \subseteq U \subseteq \hull{A}$. But from finite character of $\delta$ and the fact that the union is directed, we get $U \strong F$, so the result follows.
\item 
Let $X = \{ \Delta(ab / C) : b \in D(F) \text{ (a finite tuple)}\}$. Since $C \strong F$, as $b$ ranges over finite tuples from $D(F)$, the value of $\Delta(ab / C)$ is always in $\N$, so we can choose $b$ such that $\Delta(ab/C)$ is minimal, and for that value of $\Delta$ we can choose $b$ such that $\ldim_\Q(ab/C)$ is minimal. Then for any $d \in D(F)$ we have 
\[
\Delta(d/Cab) = \Delta(abd/C) - \Delta(ab/C) \geq 0
\]
by minimality. Hence $Cab \strong F$, and by the minimality of the linear dimension its span is $\hull{Ca}$.
\end{enumerate}
\end{proof}
The proofs in \cite{bays_pseudo-exponential_2018} work with the graph $\Gamma$ of the exponential map rather than the domain $D(F)$, and in fact work in greater generality, but the difference is not relevant for this paper. Older proofs of similar statements in \cite{kirby_finitely_2013} work under the assumption that the kernel is strongly embedded, or something similar, but this assumption is not needed.

The free extensions of Constructions~\ref{free EA construction} and~\ref{free ELA construction} are always strong. To see this, it is immediate that the one-step free extensions are strong, and then one can iterate. Furthermore, intermediate steps on the free constructions are also strong.

On the other hand, finitely generated strong extensions are very close to being free extensions, and in particular they are classifiable \cite{kirby_finitely_2013, kirby_exponentially_2014}, which gives rise to a form of stability in the type-counting sense. So stability of an independence relation as we show here is to be expected, albeit not automatic as the setting is not first-order and indeed we only prove it in the case of very full kernel.

\begin{theorem}\thlabel{strong extensions are nearly free}
Suppose $F$ is an ELA-field, and $A \strong F$ is a strong partial E-subfield of $F$. Then the ELA-closure $\gen{A}^{\ELA}_F$ of $A$ inside $F$ is also strong in $F$, and it is isomorphic to a free ELA-field extension $A^\ELA$.

Furthermore, if the hypotheses of \thref{uniqueness of free extensions} hold, then the isomorphism type of $\gen{A}^{\ELA}_F$ over $A$ does not depend on the choice of strong ELA-extension $F$.
\end{theorem}
\begin{proof}
This follows from the proof of \cite[Theorem~2.18]{kirby_finitely_2013}, which exploits the fact that the ELA-closure is the union of a chain of one-step free extensions. That theorem is stated with the assumptions (1) or (2) in \thref{rem:uniqueness-countable-case}, but those assumptions are used only in the uniqueness part of the proof. We get the uniqueness in the ``furthermore'' statement instead from \thref{uniqueness of free extensions}.
\end{proof}

It follows from \thref{strong extensions are nearly free} that for any ELA-field $F$ and subset $A$ we have  $\gen{\hull{A}_F}^{\ELA}_F$ isomorphic to $(\hull{A}_F)^{\ELA}$. To simplify notation, we write the former as $\hull{A}^{\ELA}_F$, or just $\hull{A}^{\ELA}$ without the subscript. So $\hull{A}^{\ELA}_F$ is the smallest strong ELA-subfield of $F$ containing $A \cup \ker_F$, and it follows that the category $\ELAFstrong$ has intersections.

\subsection{Free amalgamation}
\thref{prop:amalgam-base} shows that any two kernel-preserving extensions $A \leftarrow C \to B$ of ELA-fields can be amalgamated, and this can be done in many ways. We pick out a particular way to do it freely. Uniqueness of this free amalgamation is intimately connected to stability.
\begin{definition}
Let 
\begin{center}
\begin{tikzcd}
                  & F \\
A \arrow[ru] &                                    & B  \arrow[lu]                    \\
                  & C \arrow[lu] \arrow[ru]                   
\end{tikzcd}
\end{center}
be kernel-preserving inclusions of partial E-fields such that $F$ and $C$ are ELA-fields, $A \cap B = C$, and $F = \gen{A B}^{\ELA}_F$. We say that $F$ is a \emph{free amalgam} of $A$ and $B$ over $C$ if
\begin{enumerate}[label=(\roman*)]
\item $A \ind^{\td}_C B$, and
\item $F$ is a free ELA-extension of its partial E-subfield $\gen{A B}_F$. 
\end{enumerate}
\end{definition}
Given such an $A,B,C$, we can always construct a free amalgam by \thref{cor:independent-ela-amalgamation} and the $(-)^\ELA$ construction. We identify one case where it is unique.
\begin{lemma}\thlabel{uniqueness of free amalgam}
When $C$ is an ELA-field with very full kernel, the free amalgam of $A$ and $B$ over $C$ is unique up to isomorphism. That is, if $A \xrightarrow{f_1} F_1 \xleftarrow{g_1} B$ and $A \xrightarrow{f_2} F_2 \xleftarrow{g_2} B$ are free amalgams over $C$ then there is an isomorphism $\theta: F_1 \to F_2$ such that $\theta f_1 = f_2$ and $\theta g_1 = g_2$.
\end{lemma}
\begin{proof}
 Note that the inclusions of $C$ into $A$ and $B$ and the first condition above determine the square 
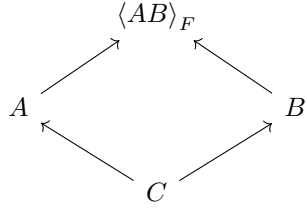
\begin{center}
\begin{tikzcd}
                  & \gen{A B}_F \\
A \arrow[ru] &                                    & B  \arrow[lu]                    \\
                  & C \arrow[lu] \arrow[ru]                   
\end{tikzcd}
\end{center}
uniquely up to isomorphism, and then from \thref{uniqueness of free extensions} we get uniqueness of the amalgam $F$ in the case where $C$ has very full kernel.
\end{proof}

We can use this construction to prove the amalgamation property for strong embeddings.
\begin{lemma}\thlabel{strong amalg lemma}
Suppose that $F$ is a free amalgam of $A$ and $B$ over $C$ as above. Suppose also that $C \strong A$. Then $B \strong \gen{AB}_F$. If also $C \strong B$ then $A \strong \gen{AB}_F$.

In particular, the category $\ELAFstrong$ has amalgamation.
\end{lemma}
\begin{proof}
This is a straightforward predimension calculation, using the fact that $A \ind^{\td}_C B$, that $C$ is an ELA-field, and that the kernel does not extend. See \cite[Proposition~5.7]{bays_pseudo-exponential_2018} for the proof in a more general setting.
\end{proof}

\begin{prop}
Each JEP-class in $\ELAFstrong$ is an AEC with amalgamation, joint embedding, and intersections.
\end{prop}
\begin{proof}
Clearly $\ELAFstrong$ is closed under isomorphisms. Coherence is well known (see e.g.\ \cite[Lemma 3.11(d)]{kirby_exponentially_2014}) and easily follows from the definition of strong embeddings. It is also straightforward to verify that we have unions of chains, using finite character of the properties involved. For downward L\"owenheim-Skolem we can, given any $A \subseteq F$, consider $\hull{A}_F^\ELA$, which will always be bounded in cardinality by $|A| + |\hull{\ker_F}_F|$, and this hull of the kernel is constant on JEP-classes. \thref{strong amalg lemma} gives amalgamation, JEP is immediate and we observed closure under intersections above, after \thref{strong extensions are nearly free}.
\end{proof}

\subsection{Closed embeddings}

Recall that the exponential algebraic closure pregeometry depends on existential information, so if $F_1 \into F_2$ is an extension of exponential fields, $\ecl_{F_1}$ may not be the restriction to $F_1$ of $\ecl_{F_2}$. Indeed $\pi$ is exponentially algebraic in $\Cexp$, because $e^{i\pi} + 1=0$ but, assuming Schanuel's conjecture, $\pi$ is actually exponentially transcendental in $\Rexp$.

\begin{definition}
An embedding $F_1 \into F_2$ of exponential fields is said to be \emph{closed} if $\ecl_{F_2}(F_1) = F_1$, or equivalently if for all $A \subseteq F_1$ we have $\ecl_{F_1}(A) = \ecl_{F_2}(A)$.
\end{definition}
It follows immediately that the independence relation $\ind^{\etd}$ satisfies \textsc{Invariance} for closed embeddings of exponential fields.

Like strong embeddings, closed embeddings can be characterised by the predimension function $\Delta$, and indeed the predimension function also characterises exponential transcendence degree.
\begin{theorem}\thlabel{etd and predim}
Let $F$ be an exponential field. Then $B$ is exponentially-algebraically closed in $F$ iff $\ker_F \subseteq B$ and for any finite tuple $a$ from $F$, not contained in $B$, we have $\Delta(a/B) \geq 1$. In particular, closed embeddings are strong embeddings.

Furthermore, if $C \strong F$ and $a$ is any finite tuple from $F$ then 
\[
\etd(a/C) = \min \{\Delta(ab/C) : b \subseteq F\}.
\]
\end{theorem}
\begin{proof}
The furthermore part is \cite[Theorem~1.3]{kirby_exponential_2010}, and the rest of the theorem follows.
\end{proof}

\section{EA-independence}

Recall that for an EA-field $F$ and a subset $A \subseteq F$ we write $\gen{A}_F^\EA$, or just $\gen{A}^\EA$ when $F$ is clear, for the smallest EA-subfield of $F$ containing $A$. 

We recall the following independence relation for EA-fields from \cite[Definition 5.1]{haykazyan_existentially_2021}. 
\begin{definition}
\thlabel{def:ea-independence}
We define $\ind^\EA$-independence as follows. Let $F$ be an EA-field and $A, B, C \subseteq F$, then:
\[
A \ind_C^{\EA, F} B \quad \Longleftrightarrow \quad \langle AC \rangle^\EA \ind_{\langle C \rangle^\EA}^{\td} \langle BC \rangle^\EA.
\]
\end{definition}
In \cite{haykazyan_existentially_2021} it was shown that this independence relation is an NSOP$_1$-like independence relation in some sense, but the list of properties proved there is not exactly the list needed for the canonicity theorem, so we explain why the extra properties also hold. We also provide a counterexample to \textsc{Base-Monotonicity}, \thref{ex:base-monotonicity-failure}, giving a direct proof that this independence relation is not simple.

\begin{proposition}
On any EA-field $F$, $\ind^\EA$ satisfies the six basic properties of an independence relation from \thref{def:independence-relation}.
\end{proposition}

\begin{proof}
All immediate from the definition or the corresponding properties of $\ind^{\td}$ and of the $\gen{-}^\EA$-closure operator.
\end{proof}
We could in fact define $\ind^\EA$ on an E-field or even a partial E-field $F$ rather than an EA-field, and prove the same result, if we relativise the EA-closure operator inside $F$. However, we will not make use of that.

We give an example to show that \textsc{Base-Monotonicity} can fail, so $\ind^\EA$ is not a simple independence relation on \EAF. 
\begin{example}
\thlabel{ex:base-monotonicity-failure}
Let $C$ be any EA-field. Let $F$ be the field $F=C(a,d,b_1,b_2)^{\alg}$, where $a,d,b_1,b_2$ are algebraically independent over $C$.
We consider various algebraically closed subfields of $F$, and will make them into EA-fields.

Let $A = C(a)^{\alg}$ and $D = C(d)^{\alg}$, and choose any exponential maps on them extending that on $C$ to make them EA-field extensions of $C$.
Let $B = D(b_1,b_2)^{\alg}$, and choose any exponential map making it an EA-field extension of $D$.

Let $t= ab_1+b_2 \in F$. Then $t$ is transcendental over $A\cup D$, and transcendental over $B$.
Let $E = A(d,t)^{\alg}$. We choose a point $u \in C(a,d)^{\alg}$ which is not in the $\Q$-linear span $A+B$, for example take $u=ad$. 
Then we can extend the exponential map from $A+B$ to an exponential map on $E$ such that  $\exp(u) = t$. 

Then we choose any exponential map on $F$ extending that on $E+B$.

Then the EA-closure of $A\cup D$ in $F$ is $E$.

We have the following diagram of EA-fields, with transcendence degrees of each extension as given.

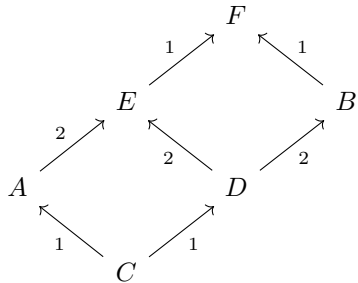
\begin{center}
\begin{tikzcd}
                  &                                    & F                                  &                    \\
                  & E \arrow[ru, "1"]                  &                                    & B \arrow[lu, "1"'] \\
A \arrow[ru, "2"] &                                    & D \arrow[ru, "2"'] \arrow[lu, "2"] &                    \\
                  & C \arrow[lu, "1"] \arrow[ru, "1"'] &                                    &                   
\end{tikzcd}
\end{center}

Now we have $C \seq D\seq B$ and by considering transcendence degrees, we see that $A\ind_C^{\td} B$ and thus $A \ind_C^\EA B$ but $E \nind_D^{\td} B$, and so $A \nind_D^\EA B$. So \textsc{Base-Monotonicity} does not hold.

\begin{remark}
This gives a good illustration of what the \textsc{Base-Monotonicity} property means. To see whether or not $\A \ind_C^{\EA} B$, we look only at $\gen{AC}^\EA \cup \gen{BC}^\EA$, not at all of $\gen{ABC}^\EA$. 
\end{remark}
\begin{remark}
We can contrast this example with the theory ACFA of (existentially closed) fields with an automorphism, $\sigma$. This is a simple theory, with simple independence relation given by $A \ind^{\mathrm{ACFA}}_C B$ if and only if $\sigma\text{-cl}(AC) \ind^{\td}_{\sigma\text{-cl}(C)} \sigma\text{-cl}(BC)$, where $\sigma\text{-cl}(X)$ means the closure of $X$ under $\sigma$, $\sigma^{-1}$, and field-theoretic algebraic closure. 

If we try to construct an example similar to \thref{ex:base-monotonicity-failure} but with $\sigma$-closed fields in place of EA-fields, we find that the field $E$, which is now the $\sigma$-closure of $A \cup D$, is just the field-theoretic algebraic closure of $A \cup D$, because the automorphism $\sigma$ commutes with the field operations. Of course as $\ind^{\mathrm{ACFA}}$ is simple it does satisfy \textsc{Base-Monotonicity}.
\end{remark}

\end{example}

We can now put together the proof that $\ind^\EA$ is an NSOP$_1$-like non-simple independence relation on the category \EAF\ of EA-fields.
\begin{proof}[Proof of \thref{thm:ea-nsop1-like-independence}]
Since \textsc{Base-Monotonicity} fails, $\ind^\EA$ is non-simple.

The $\gen{-}^\EA$-closure operator respects embeddings of EA-fields, so \textsc{Invariance} holds. 
The \textsc{Extension} property is verified in \cite[Proposition 10.5]{dobrowolski_kim-independence_2022}.
\textsc{3-amalgamation} is verified in \cite[Theorem 6.5]{haykazyan_existentially_2021}. (In fact, $n$-amalgamation is proved in \cite[Theorem 5.4]{haykazyan_existentially_2021}.)

Finally, following \cite[Remark 9.8]{dobrowolski_kim-independence_2022}, \textsc{Club Local Character} with $\lambda=\aleph_1$ follows using the same methods as in \cite{kaplan_local_2019}, because \cite[Theorem 6.5]{haykazyan_existentially_2021} actually gives us a strengthened version of \textsc{Finite Character} called \textsc{Strong Finite Character}.
\end{proof}
We note that the \textsc{Strong Finite Character} property makes use of formulas, and so this proof of \textsc{Club Local Character} makes essential use of the fact that \EAF\ is the category of models of some theory. The other proofs are more algebraic (semantic) in nature.

\section{ELA-independence}
We now come to a relation of independence which takes account of the kernel of the exponential map, and so is appropriate when we have fixed the kernel.

Recall that for an ELA-field $F$ and a subset $A \subseteq F $, we write $\gen{A}^\ELA_F$ or just $\gen{A}^\ELA$, for the smallest ELA-subfield of $F$ containing $A \cup \ker_F$.

\begin{definition}
\thlabel{def:ela-independence}
We define $\ind^\ELA$-independence as follows. Let $F$ be an ELA-field and $A, B, C \subseteq F$, then:
\[
A \ind_C^{\ELA, F} B \quad \Longleftrightarrow \quad \gen{AC}^\ELA \ind_{\gen{C}^\ELA}^{\td} \gen{BC}^\ELA.
\]
\end{definition}

\begin{proposition}
On any ELA-field $F$, $\ind^\ELA$-independence satisfies the six basic properties of an independence relation from \thref{def:independence-relation}.
Furthermore, it satisfies \textsc{Invariance}  for kernel-preserving embeddings, so is an independence relation on \ELAFkp.
\end{proposition}
\begin{proof}
The basic properties are almost immediate, as for $\ind^\EA$.
Since the ELA-closure $\gen{-}^\ELA$ is preserved under kernel-preserving embeddings of ELA-fields, the \textsc{Invariance} property holds on \ELAFkp.
\end{proof}

A variant of \thref{ex:base-monotonicity-failure} shows that \textsc{Base-Monotonicity} fails, so it is not simple.
\begin{example}
Let $C$ be an ELA-field, and take $F$ to be the ELA-extension of $C$ generated by algebraically independent elements $a,d,b_1,b_2$ subject only to the relation $\exp(ad) = ab_1+b_2$.  Now we define several $\ELA$-subfields of $F$, namely $A = \gen{Ca}^{\ELA}_F$, $D=\gen{Cd}^{\ELA}_F$, $B= \gen{Db_1,b_2}^{\ELA}_F$, $E=\gen{A\cup D}^\EA_F$. 

From the freeness of the construction we see that $A\ind_C^{\td} B$ and therefore $A \ind_C^\ELA B$. On the other hand, looking at the elements $a, ab_1+b_2, b_1, b_2$ we see that $E \nind_D^{\td} B$, and so $A \nind_D^\ELA B$. Thus, $\ind^\ELA$ does not satisfy \textsc{Base-Monotonicity}.
\end{example}

While they look similar, $\EA$-independence and $\ELA$-independence are different.
\begin{example}
We construct an ELA-field $F$ and EA-subfields $A,B,C$ with the same kernel such that 
\[A \ind_C^{\EA, F} B \quad \text{ but } \quad A \nind_C^{\ELA, F} B .\]

To do this, take any ELA-field $C$, for example $\Cexp$. Let $F \leteq C(d)^{\ELA}$ be the free ELA-extension of $C$ on a single generator $d$, as in \thref{free ELA construction}. Then $F$ has infinite transcendence degree over $C$, and the same kernel.

Let $a \leteq e^d$, $b \leteq e^{d^2}$,  $A \leteq \gen{C(a)}^{\EA}_F$ and $B \leteq \gen{C(b)}^{\EA}_F$.

Then $\gen{A}^\ELA_F = F = \gen{B}^\ELA_F$, and so $A \nind_C^{\ELA, F} B$.

However, $a$ and $b$ are algebraically independent over $C$, and so the freeness of the construction of $F$ ensures that $A \ind_C^{\EA, F} B$.
\end{example}

\begin{example}
We can also get the opposite situation. For this, let $D = \Cexp$ or any ELA-field. Then we take $F \leteq D(a,b)^{\ELA}$, the free ELA-extension on two generators, and take EA-subfields
\[A \leteq \gen{D(a,e^b)}^\EA_F, \qquad  B \leteq \gen{D(e^a,b)}^\EA_F,  \quad \text{and} \quad  C \leteq \gen{D\left(e^{e^a},e^{e^b}\right)}^\EA_F .\]
Then $\gen{C}^\ELA_F = \gen{A}^\ELA_F = \gen{B}^\ELA_F = F$, so we have $A \ind_C^{\ELA, F} B$ trivially.

However, $A \cap B = \gen{D(e^a,e^b)}_F^\EA$ which properly contains $C$, and so $A \nind_C^{\EA, F} B$.
\end{example}

We now prove that $\ind^\ELA$ is an  NSOP$_1$-like and non-simple independence relation on the category  $\ELAFkp$ of ELA-fields together with kernel-preserving embeddings, or more precisely on each connected component, which is obtained by fixing the \ELA-closure of the kernel, and is an AEC with amalgamation, joint embedding, and intersections. It remains to prove \textsc{Club Local Character}, \textsc{Extension}, and \textsc{3-Amalgamation}.

\begin{proposition}\thlabel{prop:ELA Club Local Char}
The relation $\ind^\ELA$ satisfies \textsc{Club Local Character} on each connected component of \ELAFkp. The relevant cardinal is $\lambda = \kappa^+$, where $\kappa$ is cardinality of the kernel in the connected component.
\end{proposition}
\begin{proof}
Let $F$ be an ELA-field, and $A,B\subseteq F$ with $A$ finite. We prove that the set
\[
\calC = \{ B_0 \in [B]^{< \lambda} : A \ind^{\ELA}_{B_0} B \}
\]
is club in $[B]^{< \lambda}$, where $\lambda = |\ker_F|^+$.

\textbf{Closed.} Let $(B_i)_{i < \gamma}$ with $\gamma < \lambda$ be a chain in $\calC$. Set $B_\gamma = \bigcup_{i < \gamma} B_i$. For every $i < \gamma$ we have by assumption that $\langle A B_i \rangle^\ELA \ind_{\langle B_i \rangle^\ELA}^{\td} \gen{B}^\ELA$. So by \textsc{Base-Monotonicity} for $\ind^{\td}$ we have that $\langle A B_i \rangle^\ELA \ind_{\langle B_\gamma \rangle^\ELA}^{\td} \gen{B}^\ELA$ for every $i < \gamma$. Then because $\langle A B_\gamma \rangle^\ELA = \bigcup_{i < \gamma} \langle A B_i \rangle^\ELA$ we can use \textsc{Finite Character} of $\ind^{\td}$ to conclude that $\langle A B_\gamma \rangle^\ELA \ind_{\langle B_\gamma \rangle^\ELA}^{\td} \gen{B}^\ELA$ and so indeed $A \ind^{\ELA}_{B_\gamma} B$.

\textbf{Unbounded.} Let $D \in [B]^{< \lambda}$. Then by \textsc{Local Character} and \textsc{Base-Monotonicity} of $\ind^{\td}$ there is $B_0 \subseteq \gen{B}^\ELA$ with $|B_0| < \lambda$ such that $D \subseteq B_0$ and $A \ind_{B_0}^{\td} B$. Since $|A B_0| < \lambda$ and $\lambda = |\ker_F|^+$ we have that $|\langle A B_0 \rangle^\ELA| < \lambda$.

Then by \textsc{Local Character} for $\ind^{\td}$ (or rather, by a standard consequence), there is $B_1 \subseteq \gen{B}^\ELA$ with $|B_1| < \lambda$ such that $B_0 \subseteq B_1$ and $\langle A B_0 \rangle^\ELA \ind_{B_1}^{\td} \gen{B}^\ELA$. Repeating this process we obtain a chain $(B_i)_{i < \omega}$ of subsets of $\gen{B}^\ELA$, each of cardinality $< \lambda$, such that $\langle A B_i \rangle^\ELA \ind_{B_{i+1}}^{\td} \gen{B}^\ELA$ for all $i < \omega$. Set $B_\omega = \bigcup_{i < \omega} B_i$. By \textsc{Base-Monotonicity} for $\ind^{\td}$ we have $\langle A B_i \rangle^\ELA \ind_{\langle B_\omega\rangle^\ELA}^{\td} \gen{B}^\ELA$ for all $i < \omega$. So because $\langle A B_\omega \rangle^\ELA = \bigcup_{i < \omega} \langle A B_i \rangle^\ELA$ we can use \textsc{Finite Character} for $\ind^{\td}$ to obtain $\langle A B_\omega \rangle^\ELA \ind_{\langle B_\omega \rangle^\ELA}^{\td} \gen{B}^\ELA$.

For every $c \in \langle B_\omega \rangle^\ELA$ there is some finite tuple $b_c \in B$ such that $c \in \langle b_c \rangle^\ELA$. Set $B_\omega' = D \cup \bigcup \{ b_c : c \in \langle B_\omega \rangle^\ELA \}$. Then $|B_\omega'| < \lambda$ because $|\langle B_\omega \rangle^\ELA| < \lambda$. By construction we have $D \subseteq B_\omega' \subseteq B$ while also $\langle B_\omega' \rangle^\ELA = \langle B_\omega \rangle^\ELA$. So $\langle A B_\omega' \rangle^\ELA \ind_{\langle B_\omega' \rangle^\ELA}^{\td} \gen{B}^\ELA$ and thus $A \ind^{\ELA}_{B_\omega'} B$. We conclude that $B_\omega' \in \calC$, so $\calC$ is indeed unbounded in $[B]^{< \lambda}$.
\end{proof}

This proof strategy does not seem to yield anything better than $\lambda = \kappa^+$. However, we have not proved that this is optimal, and indeed our initial guess was that one might be able to take $\lambda = \aleph_0$ for any kernel. This remains open.

\begin{proposition}\thlabel{ELA extension}
The relation $\ind^\ELA$ satisfies \textsc{Extension} on \ELAFkp.
\end{proposition}
\begin{proof}
Let $F$ be an ELA-field, let $C,B \subseteq F$, let $a$ be a possibly infinite tuple in $F$ such that $a \ind^{\ELA, F}_C B$ and let $B \subseteq D \subseteq F$. We have to produce $a'$ in some extension $N$ of $F$ such that $a' \ind_C^{\ELA, N} D$ and $\gtp(a/BC) = \gtp(a'/BC)$.

We may assume $C = \gen{C}_F^\ELA$, $B = \gen{BC}_F^\ELA$, $D = \gen{D}_F^\ELA$ and that $a$ enumerates $\gen{C a}_F^\ELA$. Let $A = \gen{B a}_F^\ELA$.

As subsets of $F$, it may be that $A$ and $D$ are not independent from each other over $B$. However, we can also regard them as extensions of $B$ and let $M$ be their free amalgam, shown by the dashed arrows in the diagram below. We now have both $F$ and $M$ as extensions of $D$, and we let $N$ be their free amalgamation, yielding the dotted arrows in the diagram below.
\[
\begin{tikzcd}
a \arrow[r]     & A \arrow[r] \arrow[rrd, dashed] & F \arrow[r, dotted]            & N                    \\
C \arrow[u] \arrow[r] & B \arrow[u] \arrow[r]           & D \arrow[u] \arrow[r, dashed] & M \arrow[u, dotted]
\end{tikzcd}
\]
We can then regard the embedding of $F$ into $N$ as an inclusion. We let $a'$ and $A'$ be the image of $a$ and $A$ in $N$, when factored through $M$. Then $A \iso A'$ with an isomorphism fixing $B$ pointwise and sending $a$ to $a'$, and so by \thref{Galois types from intersections} we have $\gtp(a'/B) = \gtp(a/B)$, which, as $C \subseteq B$, is what we needed. 

Since $M$ is the free amalgam of $A$ and $D$ over $B$, we have $\A' \ind_B^{\ELA, M} D$. Then by \textsc{Invariance} we have $A' \ind_B^{\ELA, N} D$ and by \textsc{Monotonicity} we have $a' \ind_B^{\ELA, N} D$. Also, since $a \ind^{\ELA, F}_C B$, by the above equality of Galois types and \textsc{Invariance} we have $a' \ind_C^{\ELA, N} B$. So, by \textsc{Transitivity} we find $a' \ind_C^{\ELA, N} D$, as required.
\end{proof}

\begin{proposition}\thlabel{ELA 3-amalg}
The relation $\ind^\ELA$ satisfies \textsc{3-Amalgamation} on \ELAFkp.
\end{proposition}
The proof is similar to the case of amalgamating independent systems of EA-fields and arbitrary embeddings, which was done in \cite[Theorem 5.4]{haykazyan_existentially_2021}. We will just consider 3-amalgamation, but with somewhat more complicated notation, and an inductive argument, one can also show that \ELAFkp\ has independent $n$-amalgamation for all $n \geq 3$.

\begin{proof}
Suppose we are given a commuting diagram consisting of the solid arrows below, such that $F_i \ind^{\ELA, F_{ij}}_F F_j$ for all $1 \leq i < j \leq 3$.
\[
\begin{tikzcd}[row sep=tiny]
                                   & F_{13} \arrow[rr, dashed] &                           & F'                        \\
F_1 \arrow[ru] \arrow[rr]          &                           & F_{12} \arrow[ru, dashed] &                           \\
                                   & F_3 \arrow[uu] \arrow[rr] &                           & F_{23} \arrow[uu, dashed] \\
F \arrow[ru] \arrow[rr] \arrow[uu] &                           & F_2 \arrow[uu] \arrow[ru] &                          
\end{tikzcd}
\]
We will construct $F'$ with the dashed arrows such that the entire diagram commutes, and such that $F_1 \ind^{\ELA, F'}_F F_{23}$. We will in fact additionally get $F_2 \ind^{\ELA, F'}_F F_{13}$ and $F_3 \ind^{\ELA, F'}_F F_{12}$ from the symmetry of the construction. To distinguish between the exponential maps on these fields, we will use subscripts and write, say $\exp_1$ or $\exp_{12}$, with $\exp'$ for the map on $F'$.

First, we can amalgamate the system just as algebraically closed fields, to get an algebraically closed field $F''$ and embeddings into it such that $F_1 \ind_F^{\td,F''} F_{23}$.

As in the proof of \cite[Theorem 5.4]{haykazyan_existentially_2021}, the map $\exp_{12}\cup \exp_{23} \cup \exp_{31}$ extends to a homomorphism $\exp''$ from $F_{12} + F_{13} + F_{23}$ to $(F'')^\times$, making $F''$ into a partial E-field.

We must show that there are no new kernel elements in $F_{12} + F_{13} + F_{23}$. Let $a_{12} \in F_{12}, a_{13} \in F_{13}, a_{23} \in F_{23}$ such that $\exp_{12}(a_{12}) \exp_{13}(a_{13}) \exp_{23}(a_{23}) = 1$. Write $K = \ker_F$ for the kernel of the ELA-fields in the original system, so we need to show that $a_{12} + a_{13} + a_{23} \in K$.

Using a lemma of Shelah on stable systems of models (in this case algebraically closed fields) \cite[Fact XII.2.5]{shelah_classification_1990}, also quoted as \cite[Fact~5.3]{haykazyan_existentially_2021}, we can find $c_1 \in F_1$ and $c_2 \in F_2$ such that $\exp_{12}(a_{12}) c_1 c_2 = 1$. As $F_1$ and $F_2$ are ELA-fields there are $b_1 \in F_1$ and $b_2 \in F_2$ such that $\exp_1(b_1) = c_1$ and $\exp_2(b_2) = c_2$. Hence we have $\exp_{12}(a_{12} + b_1 + b_2) = 1$ and so $a_{12} + b_1 + b_2 \in K$.

We also have $\exp_{13}(b_1) \exp_{23}(b_2) = \exp_{12}(a_{12})^{-1} = \exp_{13}(a_{13}) \exp_{23}(a_{23})$, so $\exp_{13}(a_{13} - b_1) \exp_{23}(a_{23} - b_2) = 1$. Thus we have that $\exp_{13}(a_{13} - b_1) = \exp_{23}(-(a_{23} - b_2)) \in F_{13} \cap F_{23} = F_3$. As $F_3$ is an ELA-field there is $d \in F_3$ with $\exp_3(d) = \exp_{13}(a_{13} - b_1) = \exp_{23}(-(a_{23} - b_2))$. Therefore $a_{13} - b_1 - d \in K$ and $a_{23} - b_2 + d \in K$. Since $K$ is an abelian group we get that their sum $a_{13} + a_{23} - (b_1 + b_2)$ is in $K$. Combining with $a_{12} + b_1 + b_2 \in K$ from before, we conclude that indeed $a_{12} + a_{13} + a_{23} \in K$. Hence the embeddings of the $F_{ij}$ into $F''$ are kernel-preserving.

Now we set $F' \leteq (F'')^\ELA$ to complete the system with an ELA-field. This free extension is also kernel-preserving. The system is independent with respect to $\ind^{\td}$ and each node is an ELA-subfield (with the same kernel), hence it is an $\ind^\ELA$-independent system as required.
\end{proof}

That completes the proof of \thref{ELA-indep is NSOP1}.

\section{Strong independence}
Recall that for an ELA-field $F$ and $A \subseteq F$ we write $\hull{A}_F^{\ELA}$, or just $\hull{A}^\ELA$, for the smallest strong ELA-subfield of $F$ containing $A \cup \ker_F$ and, if $F$ has very full kernel, the isomorphism type of $\hull{A}^{\ELA}_F$ does not depend on $F$ beyond the isomorphism type of $\gen{\hull{A}}_F$.

\begin{definition}
\thlabel{def:strong-independence}
Let $F$ be an ELA-field and $A, B, C \subseteq F$. We say that \emph{$A$ is strongly independent from $B$ over $C$ in $F$}, and write $A \ind^{\strong, F}_C B$, if
\begin{enumerate}[label=(STR\arabic*), widest={(STRX)}]
\item $\hull{AC}^\ELA \ind^{\td}_{\hull{C}^\ELA} \hull{BC}^\ELA$, and
\item $\hull{AC}^\ELA \cup \hull{BC}^\ELA \strong F$. 
\end{enumerate}
\end{definition}

We now show that this strong independence is related to free amalgamation and give an equivalent definition which is easier to check. 

\begin{proposition}\thlabel{strong indep equiv form}
Let $F$ be an ELA-field, let $A,B,C \subseteq F$, and for notational convenience assume that $C=\hull{C}^\ELA$, that $C \subseteq A \cap B$, and that $A = \hull{A}$ and $B=\hull{B}$.

Then $A \ind^{\strong, F}_C B$ if and only if 
\begin{enumerate}[label=(STR\arabic*$'$)]
\item $A,\exp(A) \ind^{\td}_C B,\exp(B)$, and
\item $A \cup B \strong F$. 
\end{enumerate}

Equivalently, $F$ is a strong extension of the free amalgam of $\gen{A}$ and $\gen{B}$ over $C$, or equivalently again, $\hull{AB}_F^\ELA$ is isomorphic to that free amalgam.
\end{proposition}
\begin{proof}
Suppose conditions (STR1) and (STR2) hold. Then (STR1$'$) holds by {\sc Monotonicity} (and {\sc Symmetry}) for $\ind^{\td}$.

Since $A, B \strong F$, the extensions $\gen{A} \into \hull{A}^\ELA$ and $\gen{B} \into \hull{B}^\ELA$ are free by \thref{strong extensions are nearly free}, so there are $\Q$-linear bases $(a_i)_{i<\alpha}$ of $\hull{A}^\ELA$ over $A$ and $(b_i)_{i<\beta}$ of $\hull{B}^\ELA$ over $B$ generating the chains of one-step free extensions. It follows from (STR1) that $(a_i)_{i<\alpha}$ also generates a chain of one-step free extensions of $\gen{A \cup B}$, and then that $(b_i)_{i<\beta}$ generates a chain of one-step free extensions of $\hull{A}^\ELA \cup \gen{B}$. So the extensions 
\[
\gen{A \cup B} \into \gen{\hull{A}^\ELA \cup B} \into \gen{\hull{A}^\ELA \cup \hull{B}^\ELA}
\]
are free, and hence strong. Combining with (STR2), we see that $A \cup B \strong F$, so (STR2$'$). 

Conversely, suppose (STR1$'$) and (STR2$'$) hold. From (STR2$'$) and \thref{strong extensions are nearly free}, the extension $\gen{A \cup B} \into \hull{AB}^\ELA$ is free. We can choose a chain of one-step free extensions which goes via $\gen{\hull{A}^\ELA \cup B}$, and then starting with (STR1$'$) one can prove inductively on these one-step extensions that $\hull{A}^\ELA \ind^{\td}_C B$, and then that $\hull{A}^\ELA \ind^{\td}_C \hull{B}^\ELA$, which gives (STR1). Likewise (STR2) can be proved by induction on the one-step free extensions.

It follows that conditions (STR1$'$) and (STR2$'$) are equivalent to $ \hull{AB}_F^\ELA$ being the free amalgam of $A$ and $B$ over $C$.
\end{proof}

We now verify that $\ind^\strong$ satisfies the various properties of a stable independence relation, under appropriate hypotheses.
\begin{proposition}
\thlabel{prop:strong-independence-basic-properties}
Let $F$ be any ELA-field. Then $\ind^\strong$ satisfies the six basic properties of an independence relation on $F$, and \textsc{Base-Monotonicity}.
\end{proposition}
\begin{proof}
We get \textsc{Normality}, \textsc{Existence}, \textsc{Symmetry}, and \textsc{Finite Character} directly from the definition and the corresponding properties of algebraic independence and $\hull{-}^\ELA$-closure.

For \textsc{Transitivity}, assume $A \ind^\strong_C D$ and $A \ind^\strong_D B$ with $C \subseteq D$. Condition (STR1) holds by \textsc{Transitivity} for algebraic independence. Condition (STR2) follows from a direct calculation:
\[
(\hull{AC}, \hull{BC})^\ELA = (\hull{AC}, \hull{DC}, \hull{BD})^\ELA = (\hull{AD}, \hull{BD})^\ELA = \hull{ABD}^\ELA,
\]
where the first equality follows from $C \subseteq D \subseteq B$, and the second and third from $A \ind_C^\strong D$ and $A \ind_D^\strong B$ respectively.

For \textsc{Monotonicity}, suppose $A \ind^\strong_C B$, and $D \subseteq B$. We want to show $A \ind^\strong_C D$. We may assume all of $A$, $B$, $C$, and $D$ are strong ELA-subfields of $F$, and $C \subseteq A \cap D$.

Condition (STR1$'$) follows from \textsc{Monotonicity} for $\ind^{\td}$. For condition (STR2$'$), we have $A \ind^{\td}_C B$, so by \textsc{Base-Monotonicity} and then \textsc{Normality} for $\ind^{\td}$ we have $AD \ind^{\td}_D B$, the latter being equivalent to $\gen{AD} \ind^{\td}_D B$.

We have $D \strong F$, so in particular $D \strong B$. So applying \thref{strong amalg lemma}, we get  $\gen{AD} \strong \gen{AB}$. We know $\gen{AB} \strong F$, and the composite of strong embeddings is strong, so $\gen{AD} \strong F$, which is condition (STR2$'$).
Hence $A \ind^\strong_C D$.

For \textsc{Base-Monotonicity}, suppose again that $A \ind^\strong_C B$, and $C \subseteq D \subseteq B$. We now want to show $A \ind^\strong_D B$. Again we may assume all of $A$, $B$, $C$, and $D$ are strong ELA-subfields of $F$, and $C \subseteq A \cap D$. By \textsc{Monotonicity} it suffices to prove that $\hull{AD} \ind_D^\strong B$, for which we will use \thref{strong indep equiv form}.

As in the proof of \textsc{Monotonicity}, we have $\gen{AD}  \ind^{\td}_D B$, and $\gen{AD} \strong F$, so $\hull{AD} = \linspan(AD)$. Hence $\hull{AD} \cup \exp(\hull{AD}) \subseteq \gen{AD}$ and (STR1$'$) holds. Now note that $\gen{\hull{AD} \cup B} = \gen{A \cup D \cup B} = \gen{A \cup B}$ because $D \subseteq B$, and hence $\hull{AD} \cup B \strong F$. So (STR2$'$) holds, which concludes our proof.
\end{proof}

Recall that $\ELAFstrong$ is the category of all ELA-fields with strong embeddings.
\begin{proposition}\thlabel{prop: strong invariance}
The independence notion $\ind^\strong$ satisfies \textsc{Invariance} for strong embeddings, and hence is an independence notion on the category \ELAFstrong.
\end{proposition}
\begin{proof}
Suppose $F_1 \strong F_2$ is a strong extension of ELA-fields. Then for any subset $X \subseteq F_1$ we have $\hull{X}^\ELA_{F_1} = \hull{X}^\ELA_{F_2}$. Then (dropping the subscripts), since $F_1 \strong F_2$ we also have $\hull{AC}^\ELA \cup \hull{BC}^\ELA \strong F_1$ if and only if $\hull{AC}^\ELA \cup \hull{BC}^\ELA \strong F_2$.
So the result follows.
\end{proof}

\begin{proposition}
\thlabel{prop:strong-local-character}
The independence relation $\ind^\strong$ satisfies \textsc{Local Character} on \ELAFstrong, and the cardinal $\lambda$ involved is $\aleph_0$.
\end{proposition}
\begin{proof}
Let $F$ be an ELA-field, and let $A, B \subseteq F$ with $A$ finite. We have to find a finite $B_0 \subseteq B$ such that $A \ind_{B_0}^{\strong} B$. 

First we show that we can assume $B = \hull{B}^\ELA$. If there is a finite $B_1 \subseteq \hull{B}^\ELA$ such that $A \ind_{B_1}^\strong \hull{B}^\ELA$ then by finite character of the $\hull{-}^\ELA$ operator there is a finite $B_0 \subseteq B$ with $B_1 \subseteq \hull{B_0}^\ELA$, and hence $\hull{B_0}^\ELA = \hull{B_1}^\ELA$. So then $A \ind_{B_0}^\strong B$.

Next, by \thref{hull finite char}, there is a finite $A' \supseteq A$ such that $A' B \strong F$. We can replace $A$ by $A'$, by \textsc{Monotonicity} for $\ind^\strong$, so we assume $A B \strong F$.

By \textsc{Local Character}  for $\ind^{\td}$, there is finite $B' \subseteq B$ with $A \exp(A) \ind_{B'}^{\td} B$. Let $C \leteq \hull{B'}^\ELA$. Then $C \strong AB$ so, by \thref{hull finite char} again, there is a finite $B_0 \subseteq B$ with $B' \subseteq B_0$ such that $C A B_0 \strong F$. So by \thref{strong extensions are nearly free} $\hull{CAB_0}^\ELA$ is (isomorphic to) a free ELA-extension of $\gen{CAB_0}$. This free extension can be factorised as $\gen{CAB_0} \hookrightarrow \gen{A \hull{B_0}^\ELA} \hookrightarrow \hull{CAB_0}^\ELA$, where each inclusion is free. As free extensions are strong we have $A_1 \leteq A \hull{B_0}^\ELA \strong F$.

By \textsc{Base-Monotonicity} for $\ind^{\td}$ and our choice of $B' \subseteq \hull{B_0}^\ELA \subseteq B$ we have $A \exp(A) \ind^{\td}_{\hull{B_0}^\ELA} B$. Then by \textsc{Normality} for $\ind^{\td}$, we get $A_1 \ind^{\td}_{\hull{B_0}^\ELA} B$. We also have $A \subseteq A_1 \subseteq AB$, so $\hull{A_1 B} = \hull{AB}$. Since $AB \strong F$ we thus have $A_1 B \strong F$. Hence conditions (STR1$'$) and (STR2$'$) hold, so $A_1 \ind^\strong_{B_0} B$.

Finally, $A \ind^\strong_{B_0} B$ by \textsc{Monotonicity}.
\end{proof}

\begin{proposition}
\thlabel{prop: strong extension}
The independence relation $\ind^\strong$ satisfies \textsc{Extension} on the category \ELAFstrong.
\end{proposition}
\begin{proof}
The same as in \thref{ELA extension}, only we replace $\gen{-}^\ELA$ and $\ind^\ELA$ by $\hull{-}^\ELA$ and $\ind^\strong$ respectively.
\end{proof}

\begin{proposition}
\thlabel{prop:strong vfk stationarity}
The independence relation $\ind^\strong$ satisfies \textsc{Stationarity} on the category $\ELAF_{\mathrm{vfk},\strong}$ of ELA-fields with very full kernel and strong embeddings.
\end{proposition}
\begin{proof}
Let $C \strong F$ be a strong inclusion of ELA-fields with very full kernel. Let $B \subseteq F$, and let $a_1$ and $a_2$ be possibly infinite tuples from $F$ such that $a_1 \ind_C^\strong B$ and $a_2 \ind_C^\strong B$, and $\gtp(a_1/C) = \gtp(a_2/C)$. We may assume that $B = \hull{BC}^\ELA$ and will show that $\gtp(a_1/B) = \gtp(a_2/B)$.

Using \thref{Galois types from intersections} together with $\gtp(a_1/C) = \gtp(a_2/C)$ we find an isomorphism $\theta: \hull{C a_1}^\ELA \iso \hull{C a_2}^\ELA$, fixing $C$ pointwise and sending $a_1$ to $a_2$. As $\hull{C a_i}_F^\ELA \ind_C^\strong B$ for $i=1,2$ we can apply \thref{uniqueness of free amalgam} to see that $\theta$ extends to an isomorphism $\hull{B a_1}_F^\ELA \iso \hull{B a_2}_F^\ELA$, fixing $B$ pointwise and sending $a_1$ to $a_2$. By \thref{Galois types from intersections} again we then indeed conclude that $\gtp(a_1/B) = \gtp(a_2/B)$.
\end{proof}

Putting the above results together, we can now prove that $\ind^\strong$ is a stable independence relation on $\ELAF_{\mathrm{vfk},\strong}$ (or, more correctly, on each connected component).

\begin{proof}[Proof of \thref{strong indep is stable}]
The basic properties, together with \textsc{Base-Monotonicity}, are proved in \thref{prop:strong-independence-basic-properties}. We get \textsc{Invariance} from \thref{prop: strong invariance}, \textsc{Local Character} from \thref{prop:strong-local-character}, and \textsc{Extension} from \thref{prop: strong extension}. We proved these properties for $\ind^\strong$ as an independence relation on $\ELAF_\strong$, but they are preserved when restricting to the subcategory $\ELAF_{\mathrm{vfk},\strong}$ which consists of those connected components of $\ELAF_{\strong}$ where the kernel of the exponential map is very full.  \textsc{Stationarity} is given by \thref{prop:strong vfk stationarity}. Then by Remark \ref{indep-properties-remark} we get \textsc{Club Local Character} and \textsc{3-amalgamation}, completing the list of required properties.
\end{proof}

\subsection{More general kernels}
\label{subsec:more-general-kernels}
As mentioned in the introduction, we conjecture that the restriction to exponential fields with very full kernel is not needed, and that strong independence is a stable independence relation on $\ELAF_\strong$. Only the \textsc{Stationarity} property is needed, and this is equivalent to the uniqueness of free amalgams. This in turn is related to the uniqueness of the free ELA-closure, for which we give sufficient conditions in \thref{uniqueness of free extensions} and \thref{rem:uniqueness-countable-case}. The assumption of very full kernel essentially identifies the appropriate consequence of first-order saturation to sidestep any obstacles to amalgamation (and hence the construction of isomorphisms to show uniqueness) which might occur. The alternative conditions stated in \thref{rem:uniqueness-countable-case} make use of the so-called Thumbtack Lemma of \cite{Zilber06covers, BZ11} of Kummer theory, and we have uniqueness in the case that everything is countable. In particular, we can prove the case of \textsc{Stationarity} where $a, B, C$ are all countable.
The construction of Zilber's exponential field and the proof of its uncountable categoricity in \cite{zilber_pseudo-exponentiation_2005, bays_pseudo-exponential_2018} uses a higher amalgamation technique (excellence) to extend this uniqueness from the countable case to the arbitrary uncountable cardinalities, using systems which are independent with respect to the pregeometry $\ecl$. We would hope that a similar technique could be used in our case, especially in the case of exponential fields $F$ such that $\hull{\emptyset}^\ELA_F$ is countable, but we have not been able to achieve this. The case where  $\hull{\emptyset}^\ELA_F$ is uncountable but $F$ does not have very full kernel seems harder again.

\section{Comparison with exponential algebraic independence}

Earlier we mentioned that closed embeddings can be characterised by the predimension function $\Delta$, in a similar way to strong embeddings. We use this to show that the exponential algebraic independence notion $\ind^{\etd}$ can be characterised in terms of strong independence.
Recall from the introduction:
\begin{thm}[\ref{thm:ecl-independence-vs-ecf-independence}]
Let $F$ be an exponential field and $A, B, C \subseteq F$. Then we have
\[
A \ind_C^{\etd, F} B \quad  \Longleftrightarrow \quad A \ind_{\ecl_F(C)}^{\strong, F} B.
\]
\end{thm}

\begin{proof}
We may assume $C = \ecl(C)$, $A = \hull{AC}$ and $B = \hull{BC}$. We will drop the indices for $F$ as it will not change in the proof.

First, suppose that $A \nind_C^{\etd} B$. Then there is a finite tuple $a \in A$ such that $\etd(a / B) < \etd(a / C)$. We can assume that $a$ is a basis for $\hull{B a}$ over $B$ to ensure that  $Ba \strong F$. Then since $B \strong Ba \strong F$ we have $\etd(a/B) = \Delta(a/B)$.

By \thref{etd and predim} we have $\etd(a/C) \leq \Delta(a / C) = \td(a e^a / C) - \ldim(a / C)$. 

So we have
\[
\td(a e^a / B \exp(B)) - \ldim(a / B) < \td(a e^a / C) - \ldim(a / C).
\]
Since $\ldim(a / B) \leq \ldim(a / C)$, we have that $\td(a e^a / C) > \td(a e^a / B \exp(B))$. We thus have $A \exp(A) \nind_C^{\td} B \exp(B)$ and hence $A \nind_C^\strong B$.

Conversely, suppose that $A \nind_C^{\strong} B$. 
So by \thref{strong indep equiv form} either $A \exp(A) \nind_C^{\td} B \exp(B)$ or $AB$ is not strong in $F$.

In the first case there is $a \in A$ such that $C a \strong F$ and $\td(a e^a / B \exp(B)) < \td(a e^a / C)$. There are two possibilities:
\begin{enumerate}
\item If $\ldim(a / B) < \ldim(a / C)$, then $(\linspan(C a) \cap B) \setminus C$ is nonempty and thus contains some $d \in A$. So $\etd(d / B) = 0$ and $\etd(d/C) = 1$, where the latter follows because $d \not \in C$ while $C = \ecl(C)$. Thus we have $A \nind_C^{\etd} B$.
\item If $\ldim(a / B) = \ldim(a / C)$, then $\Delta(a/B) < \Delta(a/C)$. Since $C a \strong F$ we have $\etd(a/C) = \Delta(a/C)$. So we have
\[
\etd(a/B) \leq \Delta(a/B) < \Delta(a/C) = \etd(a/C),
\]
and thus $A \nind_C^{\etd} B$.
\end{enumerate}
In the second case we assume $A \exp(A) \ind_C^{\td} B \exp(B)$ but $AB$ is not strong in $F$. So there is $a \in A$, $\Q$-linearly independent over $C$, and hence also over $B$, such that $C a \strong A$ while $Ba$ is not strong in $F$. We can then string together inequalities as follows:
\[
\etd(a / B) < \Delta(a / B) = \Delta(a / C) = \etd(a / C).
\]
The first inequality and the final equality follow from \thref{etd and predim}. The equality in the middle follows from the assumptions $\td(a e^a / B \exp(B)) = \td(a e^a / C)$, together with $\ldim(a / B) = \ldim(a / C)$. So we again conclude that $A \nind_C^{\etd} B$, which concludes the proof.
\end{proof}

\end{document}